\documentclass[12pt]{amsart}
\usepackage{amsfonts}
\usepackage{amssymb, amscd, amsthm}
\usepackage{latexsym}
\usepackage{multirow}
\usepackage{amsmath}
\usepackage[all]{xy}
\usepackage{mathrsfs}
\usepackage{amsbsy}
\usepackage{amsfonts}
\usepackage{mathrsfs}
\usepackage{verbatim}
\usepackage{version}
\usepackage{color}

\newtheorem{theorem}{Theorem}[section]
\newtheorem{lemma}[theorem]{Lemma}
\newtheorem{thm}[theorem]{Theorem}
\newtheorem{prop}[theorem]{Proposition}
\newtheorem{rem}[theorem]{Remark}
\newtheorem{coro}[theorem]{Corollary}
\newtheorem{defn}[theorem]{Definition}

\newtheorem{con/que}[theorem]{Conjecture/Question}

\newcommand{\ra}{\rightarrow}
\newcommand{\mo}{\mathcal{O}}
\newcommand{\mf}{\mathcal{F}}
\newcommand{\mg}{\mathcal{G}}
\newcommand{\ma}{\mathcal{A}}
\newcommand{\mb}{\mathcal{B}}
\newcommand{\me}{\mathcal{E}}
\newcommand{\mi}{\mathcal{I}}
\newcommand{\mk}{\mathcal{K}}

\newcommand{\ms}{\mathcal{S}}
\newcommand{\mt}{\mathcal{T}}

\newcommand{\ts}{\textbf{S}}
\newcommand{\bi}{\mathbf{I}}

\newcommand{\mq}{\mathcal{Q}}
\newcommand{\cd}{\mathcal{D}}
\newcommand{\E}{\mathscr{E}}
\newcommand{\F}{\mathscr{F}}
\newcommand{\G}{\mathscr{G}}

\newcommand{\hh}{\mathscr{H}}

\newcommand{\B}{\mathscr{B}}
\newcommand{\C}{\mathscr{C}}

\newcommand{\Proj}{\operatorname{Proj}\nolimits}

\newcommand{\Hom}{\operatorname{Hom}}
\newcommand{\Ext}{\operatorname{Ext}}
\newcommand{\Pic}{\operatorname{Pic}}

\newcommand{\Tor}{\operatorname{Tor}}
\newcommand{\SI}{\operatorname{SI}}
\newcommand{\Rep}{\operatorname{Rep}}
\newcommand{\Mat}{\operatorname{Mat}}

\def\<{\langle}
\def\>{\rangle}

\newcommand{\z}{\Theta}

\newcommand{\md}{M(rH,0)}
\newcommand{\mdh}{M(dH,0)}

\newcommand{\wrn}{M(r,0,n)}
\newcommand{\wrr}{M(r,0,r)}

\newcommand{\lcd}{\lambda_{c^r_n}(d)}

\newcommand{\lcn}{\lambda_d(c^r_n)}
\newcommand{\crn}{c^r_n}

\newcommand{\p}{\mathbb{P}}
\newcommand{\bz}{\mathbb{Z}}
\newcommand{\bc}{\mathbb{C}}
\newcommand{\km}{\mathfrak{M}}
\newcommand{\ku}{\mathfrak{U}}

\newcommand{\ks}{\mathfrak{S}}
\newcommand{\ke}{\mathfrak{E}}

\newcommand{\kv}{\mathfrak{V}}

\begin{document}
\fontsize{12pt}{14pt} \textwidth=14cm \textheight=21 cm
\numberwithin{equation}{section}
\title{Strange duality on $\mathbb{P}^2$ via quiver representations.}
\author{Yao Yuan}
\address{Yau Mathematical Sciences Center, Tsinghua University, 100084, Beijing, P. R. China}
\email{yyuan@mail.tsinghua.edu.cn; yyuan@math.tsinghua.edu.cn}

\begin{abstract} We study Le Potier's strange duality conjecture on $\mathbb{P}^2$.  We focus on the strange duality map $SD_{c_n^r,d}$ which involves the moduli space of rank $r$ sheaves with trivial first Chern class and second Chern class $n$, and the moduli space of 1-dimensional sheaves with determinant $\mathcal{O}_{\mathbb{P}^2}(d)$ and Euler characteristic 0.  By using tools in quiver representation theory, 
we show that $SD_{c^r_n,d}$ is an isomorphisms for $r=n$ or $r=n-1$ or $d\leq 3$, and in general $SD_{c^r_n,d}$ is injective for any $n\geq r>0$ and $d>0$.   

~~~

\end{abstract}

\maketitle
\tableofcontents
\section{Introduction.}
Let $X$ be a smooth projective variety over $\bc$.  Let $K(X)$ be the Grothendieck group of coherent sheaves over $X$.  Define a quadratic form $(u,c)\mapsto \chi(u\otimes c)$ on $K(X)$, where $\chi(-)$ is the holomorphic Euler characteristic and $\chi(u\otimes c)=\displaystyle{\sum_{i\geq 0}}(-1)^i\chi(\Tor^i(\mf,\mg))$ for any $\mf$ of class $u$ and $\mg$ of class $c$.  

Fix an ample divisor $H$ on $X$.  Let $c,u\in K(X)$ be orthogonal to each other with respect to $\chi(-\otimes-)$.  Let $M_X^H(c)$ and $M_X^H(u)$ be the moduli spaces of $H$-semistable sheaves of classes $c$ and $u$ respectively.  If there are no strictly semistable sheaves of classes $c$ ($u$, resp.), then over $M_X^H(c)$ ($M_X^H(u)$, resp.) there is a well-defined line bundle $\lambda_c(u)$ ($\lambda_u(c)$, resp.) called \emph{determinant line bundle associated to $u$ ($c$, resp.)}.  If there are strictly semistable sheaves of class $u$, one needs more conditions on $c$ to get $\lambda_u(c)$ well-defined (see Ch 8 in \cite{HL}).

Assume both $M^X_H(c)$ and $M^X_H(u)$ are non-empty and both $\lambda_c(u)$ and $\lambda_u(c)$ are well-defined over $M^X_H(c)$ and $M^X_H(u)$, respectively.  
The following subset of $M_X^H(c)\times M^H_X(u)$
\begin{equation}\label{introsddiv}\mathtt{D}_{c,u}:=\big\{(\mg,\mf)\in   M_X^H(c)\times M_X^H(u) \bigm| \bigoplus_i H^i(X,\mg\otimes \mf)\ne 0\big\}.\end{equation}
is a priori not always of codimension 1.  
According to \cite{LPst} (see \cite{LPst} p.9) or \S2 in \cite{Da2}, $\mathtt{D}_{c,u}$ is a divisor of $\lambda_c(u)\boxtimes\lambda_u(c)$ if the following $(\bigstar)$ is satisfied.

($\bigstar $) \emph{For all $H$-semistable sheaves $\mf$ of class $c$ and $H$-semistable sheaves $\mg$ of class $u$ on $X$, $\Tor^i(\mf,\mg)=0$, $\forall~i\ge 1$; and $H^j(X,\mf\otimes \mg)=0$, $\forall~j\geq 2$.}

Then up to scalars $\mathtt{D}_{c,u}$ induces a unique section $\sigma_{c,u}$ and a canonical map
\begin{equation}\label{inmap} SD_{c,u}:H^0(M_X^H(c),\lambda_c(u))^{\vee}\ra H^0(M_X^H(u),\lambda_u(c)).\end{equation}
Le Potier's strange duality conjecture asserts that $SD_{c,u}$ is an isomorphism.

In this paper, we let $X=\p^2$, let $c=c_n^r$ be the class of rank $r$ sheaves with trivial first Chern class and second Chern class $n$, and let $u=u_d$ be the class of 1-dimensional sheaves with determinant $\mo_{\p^2}(d)$ and Euler characteristic 0.   Then we have the strange duality map as follows.
\begin{equation}\label{sinmap} SD_{c_n^r,d}:=SD_{c_n^r,u_d}:H^0(M^{H}_{\p^2}(c_n^r),\lambda_{c_n^r}(u_d))^{\vee}\ra H^0(M^H_{\p^2}(u_d),\lambda_{u_d}(c_n^r)).\end{equation}

We prove the following two main theorems.
\begin{thm}[Theorem \ref{mainthm1}]\label{introthm1}The strange duality map $SD_{c_n^r,d}$ in (\ref{sinmap}) is an isomorphism for $r=n>0$ and $d>0$.
\end{thm}
\begin{thm}[Theorem \ref{mainthm2}]\label{introthm2}The strange duality map $SD_{c_n^r,d}$ in (\ref{sinmap}) is injective for all $n\geq r>0$ and $d>0$.
\end{thm}
Together with Proposition 4.1 and Theorem 4.16 (1) in \cite{Yuan7}, Theorem \ref{introthm1} imply the following corollary directly.
\begin{coro}[Corollary \ref{maincoro1}]\label{introcoro1}The strange duality map $SD_{c_n^r,d}$ in (\ref{sinmap}) is an isomorphism for $r>0$, $n=r+1$ and $d>0$.
\end{coro}
Together with Proposition 4.14 in \cite{Yuan7}, Theorem \ref{introthm2} imply the following corollary directly.
\begin{coro}[Corollary \ref{maincoro2}]\label{introcoro2}The strange duality map $SD_{c_n^r,d}$ in (\ref{sinmap}) is an isomorphism for $n\geq r>0$ and $d=1,2,3$.
\end{coro}

We prove Theorem \ref{introthm1} by using a famous result due to Derksen-Weyman in quiver representation theory (Theorem \ref{dewe}), which implies an analog of strange duality between moduli spaces of quiver representations (Theorem \ref{mainquiver}).  
We want to relate our map $SD_{c^r_n,d}$ to the analogous map $SD(Q)$ in quiver representation theory.
Theorem 2 in \cite{Dr1} already says that $M_{\p^2}^H(c_n^n)$ is isomorphic to some moduli spaces of representations of some quiver.  We show that $M_{\p^2}^H(u_d)$ is birational to some moduli spaces of representations of the same quiver as $M_{\p^2}^H(c_n^n)$.  Then we  prove that the birational equivalences of moduli spaces induce isomorphisms the global section spaces of determinant line bundles.

After Theorem \ref{introthm1} is proved, we prove Theorem \ref{introthm2} by generalizing the method in \cite{Yuan7}. 

In the main part of the paper, we usually write $M(r,0,n)$ ($M(dH,0)$, $\lambda_{d}(c_n^r)$, $\lambda_{c_n^r}(d)$, resp.) instead of $M^H_{\p^2}(c_n^r)$ ($M^H_{\p^2}(u_d)$, $\lambda_{u_d}(c_n^r)$, $\lambda_{c_n^r}(u_d))$, resp.). 
 
The structure of the paper is arranged as follows.  In \S2 we give some background materials on quiver representations and tilting theory.  In \S3 we list the notations we will use in next sections.  In \S4 we mainly recall some results of Dr\'ezet in \cite{Dr1} which provides an isomorphism between $M_{\p^2}^H(c_n^n)$ to a moduli space $M(Q,(n,2n))$ of representations with dimension vector $(n,2n)$ of some quiver $Q$.  In \S5
we build the birational equivalence between $M^H_{\p^2}(u_d)$ with $M(Q,(d,d))$, i.e. a moduli space of representations with dimension vector $(d,d)$ of  quiver $Q$.  Finally in \S6 we prove the main theorems.

Strange duality was at first conjectured for curves by Beaville (\cite{Bea}) and Dongai-Tu (\cite{DT}) in 1990s, and it has been proved true for ten years (\cite{Bel1},\cite{Bel2},\cite{MO1}).  For smooth projective variety of higher dimension, the conjecture in general can not be formulated.  But for surfaces, besides Le Potier's formulation for rational surfaces which has been studied for instance in \cite{Abe}, \cite{Abe2}, \cite{Da2}, \cite{GY}, \cite{Yuan1}, \cite{Yuan5}, \cite{Yuan6} and \cite{Yuan7} before, there is also a formulation due to Marian-Oprea (see \cite{MO2}) for K3 and abelian surfaces, in which a lot of results has obtained by the Marian-Oprea team (\cite{BMOY}, \cite{MO3}, \cite{MO4}, \cite{MO5}). 
\section{Preliminaries.} 

\subsection{Semi-invariants of quivers}\label{semi}
Let $\Bbbk$ be the base field which is algebraically closed.  A quiver $Q$ is a pair $Q=(Q_0,Q_1)$ consisting of the set of vertices $Q_0$ and the set of arrows $Q_1$.  Denote by $ta$ and $ha$ the tail and the head respectively of each arrow $a\in Q_1$.  A \emph{representation} $V$ of $Q$ is a family of finite dimensional $\Bbbk$-vector spaces $\{V(x)|x\in Q_0\}$ and of $\Bbbk$-linear maps $V(a):V(ta)\ra V(ha)$.  Denote by $\Gamma$ the space of integer-valued functions on $Q_0$.  The dimension vector $\underline{d}(V)$ of a representation $V$ is defined by $\underline{d}(V)(x)=\dim_{\Bbbk} V(x)$.  Then $\underline{d}(V)\in\Gamma$.  The Euler product on $\Gamma$ is defined as follows.
\begin{equation}\label{debifm}\<\alpha,\beta\>=\sum_{x\in Q_0}\alpha(x)\beta(x)-\sum_{a\in Q_1}\alpha(ta)\beta(ha).\end{equation}
Notice that $\<-,-\>$ is not symmetric.

We assume $Q$ has no oriented cycles.  Let $\Rep(Q)$ be the set of all representations of $Q$, and $\Rep(Q,\alpha)$ of those with dimension vector $\alpha$. 
Then
$$\Rep(Q,\alpha)\cong\bigoplus_{a\in Q_1}\Hom(\Bbbk^{\alpha(ta)},\Bbbk^{\alpha(ha)}).$$
The two groups 
$$GL(Q,\alpha):=\prod_{x\in Q_0}GL(\alpha(x),\Bbbk)\supset SL(Q,\alpha):=\prod_{x\in Q_0}SL(\alpha(x),\Bbbk)$$
act on $\Rep(Q,\alpha)$ such that $\forall \displaystyle{\prod_{x\in Q_0}}g(x)\in GL(Q,\alpha),~V\in\Rep(Q,\alpha)$,
$$(\displaystyle{\prod_{x\in Q_0}}g(x))\circ V=\{V(x),x\in Q_0;~g(ha)\circ V(a)\circ g(ta)^{-1}:V(ta)\ra V(ha),a\in Q_1\}.$$ 
The above actions induce actions of $GL(Q,\alpha)$ and $SL(Q,\alpha)$ on the ring $\Bbbk[\Rep(Q,\alpha)]$ of regular functions on $\Rep(Q,\alpha)$.  The ring of \emph{semi-invariants} $\SI(Q,\alpha):=\Bbbk[\Rep(Q,\alpha)]^{SL(Q,\alpha)}$ has a weight space decomposition
\[\SI(Q,\alpha)=\bigoplus_{\sigma\in\Gamma^{*}}\SI(Q,\alpha)_{\sigma},\]
where $\Gamma^{*}:=\Hom(\Gamma,\mathbb{Z})$ and
$$\SI(Q,\alpha)_{\sigma}:=\{f\in k[\Rep(Q,\alpha)]\big| (\displaystyle{\prod_{x\in Q_0}}g(x))(f)=\displaystyle{\prod_{x\in Q_0}}det(g(x))^{\sigma(e_x)}\cdot f\}$$
with $\Gamma\ni e_x(y)=\left\{\begin{array}{l}1\quad\text{   if }x=y,\\0\quad\text{   otherwise.} \end{array}\right.$.

For any $V,W\in \Rep(Q)$ we have the following exact sequence 
\begin{equation}\label{qext1}\xymatrix@C=0.5cm{0\ar[r]&\Hom_Q(V,W)\ar[r]^{\imath~~~~\qquad}&\bigoplus_{x\in Q_0}\Hom(V(x),W(x))\\ &\ar[r]^{d^V_W~~~\qquad\qquad}&\bigoplus_{a\in Q_1}\Hom(V(ta),W(ha))\ar[r]^{\qquad\qquad p}&\Ext_{Q}(V,W)\ar[r]&0.}
\end{equation} 
The map $d^V_W$ is given by 
$$\{f(x)\}_{x\in Q_0}\mapsto\{f(ha)V(a)-W(a)f(ta)\}_{a\in Q_1}.$$ 
If $\<\underline{d}(V),\underline{d}(W)\>=0$, then the map $d^V_W$ in (\ref{qext1}) will be a square matrix.  
Let $\alpha,\beta$ be two dimension vectors such that $\<\alpha,\beta\>=0$.  
We then can define a function $c$ on $\Rep(Q,\alpha)\times \Rep(Q,\beta)$ such that $c(V,W)=\det(d^V_W)$ for every $V\in \Rep(Q,\alpha)$ and $W\in \Rep(Q,\beta)$.  For any fixed $V\in \Rep(Q,\alpha)$, the restriction of $c$ to $\{V\}\times\Rep(Q,\beta)$ defines an element $c^{V}\in\SI(Q,\beta)$ with weight $\<\alpha,-\>$.  Also for any fixed $W\in\Rep(Q,\beta)$ we have $c_W\in\SI(Q,\alpha)_{-\<-,\beta\>}$ defined in analogous way.  By the result in \cite{DW} (Theorem 1 and Corollary 1) we have
\begin{thm}[Derksen-Weyman]\label{dewe}$\SI(Q,\alpha)$ is a $\Bbbk$-linear span of semi-invariants $c_W$ with $\<\alpha,\underline{d}(W)\>=0$ and the analogous result is true for the semi-invariants $c^V$.  In particular $\dim_k\SI(Q,\alpha)_{-\<-.\beta\>}=\dim_k\SI(Q,\beta)_{\<\alpha,-\>}$ for all dimension vectors $\alpha,\beta$ such that $\<\alpha,\beta\>=0$.
\end{thm}

\subsection{Stability of quivers.}
Fix a weight $\sigma\in\Gamma^*$ and a dimension vector $\alpha$ such that $\sigma(\alpha)=0$.  Let $V\in\Rep(Q,\alpha)$.  A \emph{subrepresentation} $U\subset V$ of $V$ consists of $\Bbbk$-vector subspaces $\{U(x)\subset V(x)|x\in Q_0\}$ such that $V(a)(U(ta))\subset U(ha)$ for all $a\in Q_1$, and of $\Bbbk$-linear maps $U(a):=V(a)|_{U(ta)}:U(ta)\ra U(ha)$.  We say a representation $V\in\Rep(Q,\alpha)$ is \emph{(semi)stable} with respect to weight $\sigma$ if $\forall~U\subset V$, $\sigma(\underline{d}(U))<(\leq)0$. 

Let $\Rep(Q,\alpha)_{\sigma}^{s}$ ($\Rep(Q,\alpha)_{\sigma}^{ss}$) be the subspace of $\Rep(Q,\alpha)$ of (semi)stable representations with respect to weight $\sigma$.  By \S3 in \cite{King1}, we have the following thoerem.
\begin{thm}[King]\label{GITK}$\Rep(Q,\alpha)^{ss}_{\sigma}$ is an open subvariety in $\Rep(Q,\alpha)$ and it admits a categorical (GIT) quotient $M(Q,\alpha)_{\sigma}:=\Rep(Q,\alpha)_{\sigma}^{ss}\slash\slash GL(Q,\alpha)$ which is a projective variety.  The quotient $M(Q,\alpha)_{\sigma}$ contains a smooth open subvariety $M(Q,\alpha)_{\sigma}^s$ which is a geometric quotient $\Rep(Q,\alpha)_{\sigma}^{s}//GL(Q,\alpha)$.  \end{thm}

Hence we know that $M(Q,\alpha)_{\sigma}\cong\Proj(\oplus_{n\in\mathbb{N}}\SI(Q,\alpha)_{n\sigma})$.  
\begin{rem}\label{irre1}$M(Q,\alpha)_{\sigma}$ is irreducible because so is $\Rep(Q,\alpha)^{ss}_{\sigma}$ as an open subvariety of $\Rep(Q,\alpha)\cong\Bbbk^{\displaystyle{\Sigma_{a\in Q_1}}\alpha(ta)\alpha(ha)}$
\end{rem}

\subsection{Strange duality on quiver representations.}\label{sdqr} 
Let $\alpha,\beta\in \Gamma$ be two dimension vectors such that $\<\alpha,\beta\>=0$.  We have defined the function $c$ on $\Rep(Q,\alpha)\times\Rep(Q,\beta)$ in $\S$ \ref{semi}.  The restriction function $c_W$ ($c^V$ resp.) on $\Rep(Q,\alpha)^{ss}_{-\<-,\beta\>}\times\{W\}$ ($\{V\}\times\Rep(Q,\beta)^{ss}_{\<\alpha,-\>}$, resp.) gives a section of determinant line bundle $\widetilde{\lambda}(Q,\alpha)_{-\<-,\beta\>}$ ($\widetilde{\lambda}(Q,\beta)_{\<\alpha,-\>}$, resp.) over $\Rep(Q,\alpha)^{ss}_{-\<-,\beta\>}$ ($\Rep(Q,\beta)^{ss}_{\<\alpha,-\>}$, resp.), which descends to a line bundle $\lambda(Q,\alpha)_{-\<-,\beta\>}$ ($\lambda(Q,\beta)_{\<\alpha,-\>}$, resp.) over $M(Q,\alpha)_{-\<-,\beta\>}$ ($M(Q,\beta)_{\<\alpha,-\>}$, resp.).  Moreover $c$ restricted to $\Rep(Q,\alpha)^{ss}_{-\<-,\beta\>}\times\Rep(Q,\beta)^{ss}_{\<\alpha,-\>}$ also descends to $M(Q,\alpha)_{-\<-,\beta\>}\times M(Q,\beta)_{\<\alpha,-\>}$ and gives a section $\bar{c}$ of $\lambda(Q,\alpha)_{-\<-,\beta\>}\boxtimes\lambda(Q,\beta)_{\<\alpha,-\>}$, which induces a map (the analog of the strange duality map on quiver representations)
\begin{equation}\label{sdquiver}SD(Q):H^0(M(Q,\alpha)_{-\<-,\beta\>},\lambda(Q,\alpha)_{-\<-,\beta\>})^{\vee}\ra H^0(M(Q,\beta)_{\<\alpha,-\>},\lambda(Q,\beta)_{\<\alpha,-\>}).
\end{equation}

\begin{thm}\label{mainquiver}The map $SD(Q)$ in (\ref{sdquiver}) is an isomorphism.
\end{thm}
\begin{proof}By the construction of $M(Q,\alpha)_{\sigma}$, we know that 
$$H^0(M(Q,\alpha)_{-\<-,\beta\>},\lambda(Q,\alpha)_{-\<-,\beta\>})\cong \SI(Q,\alpha)_{-\<-,\beta\>};$$
and 
$$H^0(M(Q,\beta)_{\<\alpha,-\>},\lambda(Q,\beta)_{\<\alpha,-\>})\cong\SI(Q,\beta)_{\<\alpha,-\>}.$$  

By Lemma 1 in \cite{DW} and Theorem \ref{dewe}, we see that $\SI(Q,\alpha)_{-\<-,\beta\>}$ is a $\Bbbk$-linear span of semi-invariants $c_W$ with $W\in\Rep(Q,\beta)$ semistable with respect to weight $\<\alpha,-\>$ and $c_W$ only depend on the $S$-equivalence classes of $W$; and also the analogous result is true for $\SI(Q,\beta)_{\<\alpha,-\>}$.  Hence the theorem follows from the definition of $SD(Q)$ and basic linear algebra.
\end{proof}

\subsection{Some tilting theory.}
Tilting theory helps to relate semistable sheaves to semistable quivers.  In this subsection, we recall some definitions and results in tilting theory, for more details we refer to \cite{HP}, \cite{Ba} and \cite{King2}.

\begin{defn}\label{deexcept}A coherent sheaf $\me$ on a smooth algebraic $\Bbbk$-variety $X$ is called \textbf{exceptional} if $\Hom(\me,\me)\cong \Bbbk$ and $\Ext^i(\me,\me)=0$ for all $i\geq1$.
An ordered collection $\me_1,\cdots,\me_n$ of exceptional sheaves is called \textbf{an exceptional sequence} if 
$\Ext^i(\me_j,\me_l)=0$ for all $i$ and $j>l$.  If an exceptional sequence generates the derived category $D^b(X)$ of bounded complexes 
of coherent sheaves, then it is called \textbf{full}.  A \textbf{strongly exceptional} sequence is an exceptional sequence such that $\Ext^i(\me_j,\me_l)=0$ for all $i\geq1$ and all $j,l$.
\end{defn}
\begin{defn}\label{detilting}A \textbf{tilting sheaf} is a coherent sheaf $\mt$ on $X$ such that
\begin{itemize}\item[(i)]$\Ext^i(\mt,\mt)=0$ for $i\geq1$;
\item[(ii)]$\Hom(\mt,\mt)$ has finite dimension; 
\item[(iii)]$\mt$ generates $D^b(X)$.
\end{itemize}
\end{defn}

We have 
\begin{thm}[Theorem 2.2 in \cite{King2}]\label{tilting1}Let $\mt$ be a tilting sheaf, and let $A:=\Hom(\mt,\mt)$ and $D^b(A)$ be the derived category of bounded complexes of finite dimensional $A$-right modules.  Then we have the following two derived functors
\[R^{\bullet}\Hom(\mt,-):D^b(X)\ra D^b(A),\]
and
\[-\overset{\textbf{L}~}{\otimes_{A}}\mt:D^b(A)\ra D^b(X),\]
which are mutually inverse equivalences between $D^b(X)$ and $D^b(A)$.
\end{thm}
By Lemma 4.2 and Remark 4.4 in \cite{King2} we have
\begin{lemma}\label{titlting2}If $\me_1,\cdots,\me_n$ is a strongly exceptional sequence which is also full, then $\me_1\oplus\cdots\oplus \me_n$ is a titling sheaf.
\end{lemma}
\section{Notations.} 
\begin{enumerate}
\item From now on we fix the base field $\Bbbk=\mathbb{C}$.
\item Let $X=\p^2$ and $H$ be the hyperplane class.  Let $K(\p^2)$ be the Grothendieck group of coherent sheaves over $\p^2$. 
\item Let $\mf$, $\mg$ be two coherent sheaves.  Then
\begin{itemize}
\item Denote by $r(\mf)$, $c_i(\mf)$ and $\chi(\mf)$ the rank, the i-th Chern class and the Euler characteristic of $\mf$ respectively;
\item For $r(\mf)>0$, we define the slop of $\mf$ 
 $$\mu(\mf):=\frac{c_1(\mf).H}{r(\mf)}$$
and its discriminant 
$$\Delta(\mf):=\frac{1}{r(\mf)}(c_2(\mf)-(1-\frac1{r(\mf)})\frac{c_1(\mf)^2}{2}).$$
We also denote $\Delta(\mf)$ by $\Delta(r(\mf),c_1(\mf).H,c_2(\mf))$.

\item $h^i(\mf)=dim~H^i(\mf)$ and hence $\chi(\mf)=\sum_{i\geq0}(-1)^i h^i(\mf)$;
\item $\text{ext}^i(\mf,\mg)=\dim~\Ext^i(\mf,\mg)$, $\text{hom}(\mf,\mg)=\dim~\Hom(\mf,\mg)$ and $\chi(\mf,\mg)=\sum_{i\geq0}(-1)^i\text{ext}^i(\mf,\mg)$.  
\end{itemize}
\item We denote by $u_d$ the class of 1-dimensional sheaves of determinant $dH$ and Euler characteristic 0 in $K(\p^2)$.  Let $\mdh$ be the moduli space of semistable sheaves of class $u_d$.
\item We denote by $c^r_n$ the class of sheaves of rank $r$, first Chern class 0 and second Chern class $n$ in $K(\p^2)$.  Let $M(r,0,n)$ be the moduli space of semistable sheaves of class $c^r_n$.
\item Let $M$ be any moduli space of semistable sheaves.  We write $\mf\in M$ if the $S$-equivalence class of $\mf$ is in $M$.
\end{enumerate}

\section{Height zero moduli spaces of semistable sheaves on $\p^2$.}


It is easy to see that $\Delta(\me)=\frac{1}{2r(\me)^2}(-\chi(\me,\me))+\frac12$ and by \cite{DLP} $\me$ is exceptional iff $\me$ is a stable vector bundle with $\Delta(\me)<\frac12$.   Let $\ke$ be the set of all slops of exceptional bundles.  For each element $a\in\ke$ there is exactly one exceptional bundle $\me_{a}$ up to isomorphisms such that $\mu(\me_a)=a$.  We define the interval $I_a:=(a-x_a,a+x_a)$ with $x_a=\frac32-\sqrt{\frac94-\frac1{r(\me_a)^2}}$.  Let $P(y)=\frac{y^2+3y+2}2$ be a polynomial in $y$.  Then $x_a$ is the smaller solution of the equation $P(-y)-\Delta(\me_a)=\frac12$.  

We define $r_a:=r(\me_a),~\Delta_a:=\Delta(\me_a)$ and $a.b:=\frac{a+b}2+\frac{\Delta_b-\Delta_a}{3+a-b}$ for $a,b\in\ke$.  There is a bijection $\epsilon:\mathbb{Z}[\frac12]\ra \ke$ defined inductively by setting $\epsilon(n)=n$ for $n\in\mathbb{Z}$ and
$$\epsilon(\frac{2p+1}{2^q})=\epsilon(\frac{p}{2^{q-1}}).\epsilon(\frac{p+1}{2^{q-1}}).$$
 
By \cite{Ba} and \cite{GR}, we have
\begin{thm}\label{tilting3}$\me_1,\me_2,\me_3$ is a full strongly exceptional sequence if and only if the slops $(\mu(\me_1),\mu(\me_2),\mu(\me_3))$ are of the forms \footnotesize
\[(\epsilon(\frac{p-1}{2^q}),\epsilon(\frac{p}{2^q}),\epsilon(\frac{p+1}{2^q})),~(\epsilon(\frac{p}{2^q}),\epsilon(\frac{p+1}{2^q}),\epsilon(\frac{p-1}{2^q}+3)),~(\epsilon(\frac{p+1}{2^q}-3),\epsilon(\frac{p-1}{2^q}),\epsilon(\frac{p}{2^q})).\]\normalsize
\end{thm}
We recall some results from \cite{Dr1} as follows.  
\begin{thm}(Theorem 1 in \cite{Dr1})\label{hzero1} \begin{enumerate}
\item $I_a$ are all disjoint and $\mathbb{Q}=\mathbb{Q}\cap\displaystyle{\bigcup_{a\in\ke}}I_a$;
\item There is a function $\delta:\mathbb{Q}\ra\mathbb{Q}$ defined by the formula
\[\delta(\mu)=P(-|\mu-a|)-\delta(a),\qquad if~\mu\in I_a.\]
\item The moduli space $M(r,c_1,c_2)$ of semi-stable sheaves with rank $r\geq1$, 1st Chern classes $c_1H$ and 2nd Chern class $c_2$ has positive dimension iff 
\[\delta(\frac{c_1}r)\leq \Delta(r,c_1,c_2).\]
This property also characterizes the function $\delta$.
\end{enumerate}
\end{thm}
For any $\mu\in\mathbb{Q}$, we call $a$ the \emph{associated exceptional slope} to $\mu$ if $\mu\in I_a$.  Let $a$ be the associated exceptional slope to $\frac{c_1}r$.  Then the \emph{height} $h(M(r,c_1,c_2))$ of the moduli space $M(r,c_1,c_2)$ is defined as follows.
\begin{equation}\label{delight}h(M(r,c_1,c_2)):=rr_a(\Delta(r,c_1,c_2)-\delta(\frac{c_1}{r})),
\end{equation}
By a direct investigation, we see that 
\begin{equation}\label{exht}h(M(r,c_1,c_2))=\left\{\begin{array}{ll}-\chi(\me_a,\mf)&\text{if }\mu\leq a;\\ -\chi(\mf,\me_a)&\text{if }\mu\geq a.\end{array}\right.
\end{equation}where $\mf$ is any coherent sheaf of rank $r$, 1st Chern classes $c_1H$ and 2nd Chern class $c_2$. 

Let $h(M(r,c_1,c_2))=0$, i.e. $M(r,c_1,c_2)$ is of height zero.  With no loss of generality, we assume $\frac{c_1}r\in(a-x_a,a]$.  We take a fully strongly exceptional sequence $(\me_1,\me_2,\me_3)$ such that $\mu(\me_3)=a$.  Then as discussed in \cite{Dr1}, every $\mf\in M(r,c_1,c_2)$ has the following resolution
\begin{equation}\label{resov}0\ra \me_1\otimes\Ext^1(\mf,\me_1)^{\vee}\ra\me_2\otimes\Ext^1(\ms,\mf)\ra\mf\ra0,
\end{equation}
where $\ms$ is the cokernel of the injective canonical evaluation map 
\[ev^{\vee}:\me_2\ra\me_3\otimes\Hom(\me_2,\me_3)^{\vee}.\]

Let $m_1:=\text{ext}^1(\mf,\me_1)$, $m_2:=\text{ext}^1(\ms,\mf)$ and $q:=\text{hom}(\me_1,\me_2)$ (actually $q=3r_a$, c.f. \cite{Dr1}).  The resolution in (\ref{resov}) assigns to every semistable sheaf $\mf$ a representation with dimension vector $\alpha:=(m_1,m_2)$ of the following quiver
\begin{equation}\label{gequiver}Q(q):\quad\xymatrix@C=0.5cm{x_1\ar@/^2pc/[rr]^{a_1}\ar@/^1pc/[rr]^{a_2}\ar@/_1pc/[rr]^{a_{q-1}}\ar@/_2pc/[rr]_{a_q}&\vdots&x_2,}\end{equation}
i.e. $Q(q)=(Q_0,Q_1)$ with $Q_0=\{x_1,x_2\}$ and $Q_1=\{a_1,a_2,\cdots,a_q\}$.  

Since $Q_0$ consists of two points, all $0\neq\sigma\in\Gamma^*$ such that $\sigma(\alpha)=0$ are proportional and hence there is only one stability condition $\sigma$ (up to scalars) such that $\Rep(Q(q),\alpha)^{ss}_{\sigma}$ is not empty.  Let $M(Q(q),(m_1,m_2)):=M(Q(q),\alpha)_{\sigma}$ be the unique non-empty moduli space of semistable representations of $Q$ with dimension vector $\alpha=(m_1,m_2)$.  Then we have 
\begin{thm}(Theorem 2 in \cite{Dr1})\label{hzero2}Let $M(r,c_1,c_2)$ be a moduli space of height zero.  Then resolution in (\ref{resov}) gives an isomorphism 
\begin{equation}\label{isom1}f:M(r,c_1,c_2)\ra M(Q(q),(m_1,m_2))\end{equation} 
inducing an isomorphism on the open subspaces consisting of stable objects.
\end{thm}
\begin{thm}(Theorem 3 in \cite{Dr1})\label{hzero3}There is a natural isomorphism 
\begin{equation}\label{isom2}M(Q(q),(m_1,m_2))\xrightarrow[\cong]{g} M(Q(q),(m_2,qm_2-m_1))\end{equation}
inducing an isomorphism on the open subspaces consisting of stable objects.
\end{thm}
The proof of Theorem \ref{hzero3} can be found in Chapter III in \cite{Dr1}.  For later use, we want to explain explicitly how to define the map $g$ in (\ref{isom2}).  We actually define a map $\tilde{g}$ sending $GL(Q(q),(m_1,m_2))$-orbits in $\Rep(Q(q),(m_1,m_2))^{ss}$ to $GL(Q(q),(m_2,qm_2-m_1))$-orbits in $\Rep(Q(q),(m_2,qm_2-m_1))^{ss}$.  

Let $V\in \Rep(Q(q),(m_1,m_2))^{ss}$, then $V$ can be viewed as an element in $\Hom( \mathbb{C}^{m_1},\mathbb{C}^{m_2})^{\oplus q}\cong\Hom(\mathbb{C}^q\times \mathbb{C}^{m_1},\mathbb{C}^{m_2})$.   Denote by $f_V$ the element in $\Hom( \mathbb{C}^{m_1},\mathbb{C}^q\times\mathbb{C}^{m_2})$ corresponding to $V$.  Since $V$ is semistable, then $f_V$ must be injective (Lemma 18 in \cite{Dr1}).  Denote by $C(V)$ the cokernel of the map $f_V$, then $C(V)\cong \mathbb{C}^{qm_2-m_1}$.  The projection $\mathbb{C}^q\times\mathbb{C}^{m_2}\ra C(V)$ can be viewed as an element in $\Hom(\mathbb{C}^{m_2},\mathbb{C}^{qm_2-m_1})^{\oplus q}$ hence an element $\overline{V}$ in $\Rep(Q(q),(m_2,qm_2-m_1))$, with an ambiguity caused by choosing basis of $C(V)$.  The semistability of $\overline{V}$ is stated by Lemma 19 in \cite{Dr1}.  

To be more precise, let $e_1,\cdots,e_q$ be a basis of $\mathbb{C}^q$, let $V$ be represented by $m_1\times m_2$ matrices $A_1,\cdots,A_q$ and let $\overline{V}$ be represented by $m_2\times (qm_2-m_1)$ matrices $\widetilde{A}_1,\cdots,\widetilde{A}_q$.  Then we have
$\exists ~P\in GL(qm_2)$ such that
\begin{equation}\label{compare1}\begin{pmatrix}A_1,&\cdots, &A_q\\ e_1\bi_{m_2}, & \cdots,&e_q\bi_{m_2}\end{pmatrix}\cdot P=\begin{pmatrix}\bi_{m_1},&\mathbf{0}_{m_1\times (qm_2-m_1)}\\ *, & \Sigma_{i=1}^qe_i\cdot \widetilde{A}_i\end{pmatrix},
\end{equation}
where $\bi_{m}$ is the $m\times m$ identity matrix, $\mathbf{0}_{m\times l}$ is the $m\times l$ zero matrix and $*$ stands for any matrix with compatible order.  Easy to see relation in (\ref{compare1}) defines a map $\tilde{g}$ sending $GL(Q(q),(m_1,m_2))$-orbits in $\Rep(Q(q),(m_1,m_2))^{ss}$ to $GL(Q(q),(m_2,qm_2-m_1))$-orbits in $\Rep(Q(q),(m_2,qm_2-m_1))^{ss}$.

Analogously, for any $\overline{V}\in \Rep(Q(q),(m_2,qm_2-m_1))^{ss}$, we get an element in $\Hom(\mathbb{C}^{m_2},\mathbb{C}^{qm_2-m_1})^{\oplus q}$ which  can be viewed as a map $f_{\overline{V}}:\mathbb{C}^q\times\mathbb{C}^{m_2}\ra \mathbb{C}^{qm_2-m_1}$.  
$f_{\overline{V}}$ has to be surjective by semistability of $\overline{V}$.  We define the inverse image $V=g^{-1}(\overline{V})$ to be the map $\ker(f_{\overline{V}})\hookrightarrow \mathbb{C}^q\times\mathbb{C}^{m_2}$.  It is easy to see that the relation inverse to (\ref{compare1}) is as follows.
 \begin{equation}\label{compare11}\widetilde{P}\cdot\begin{pmatrix}\widetilde{A}_1,&e_1\bi_{m_2}\\ \vdots &\vdots\\ \widetilde{A}_q,& e_q\bi_{m_2}\end{pmatrix}=\begin{pmatrix}\bi_{(qm_2-m_1)},&*\\ \mathbf{0}_{m_1\times (qm_2-m_1)}, & \Sigma_{i=1}^qe_i\cdot A_i\end{pmatrix},
\end{equation} 
where $\widetilde{P}\in GL(qm_2)$.

In particular, if $c_1=0$, then $\mu=0$ and the associated exceptional slope to $\mu$ is also $0$.  Moreover the
exceptional sequence associated to slop 0 can be taken as $(\mo_{\p^2}(-2),\mo_{\p^2}(-1),\mo_{\p^2})$.  By direct computation we have $M(r,0,n)$ is of height zero iff $n=r$, and for this case $q=3$, $m_1=r$ and $m_2=2r$.  We have a quiver specified as follows
\begin{equation}\label{spequiver}Q:=Q(3):\quad\xymatrix@C=1cm{x_{-2}\ar[r]^{y}\ar@/^1pc/[r]^{x}\ar@/_1pc/[r]^{z}&x_{-1},}\end{equation}
By Theorem \ref{hzero2} the moduli space $M(r,0,r)\xrightarrow[\cong]{f} M(Q,(r,2r))$.  By Theorem \ref{hzero3} we have an isomorphism $M(Q,(r,r))\xrightarrow[\cong]{g} M(Q,(r,2r))$. 

\section{Moduli spaces of 1-dimensional semi-stable sheaves.}
Recall that $\md$ is the moduli space of 1-dimensional semistable sheaves on $\p^2$ with determinant $rH$ and Euler characteristic 0.  We will see in this section that there are two birational maps $\Psi: \md\dashrightarrow M(Q,(r,r))$ and $\Phi: \md\dashrightarrow M(r,0,r)$ such that the following diagram commutes
\begin{equation}\label{comdia1}\xymatrix@C=1cm{\md\ar@{-->}[r]^{\Psi\quad}\ar@{-->}[d]_{\Phi}& M(Q,(r,r))\ar[d]^{g}_{\cong}\\ M(r,0,r)\ar[r]_{f\quad}^{\cong\quad}&M(Q,(r,2r)),}
\end{equation}
where $f,~g$ are defined at the end of the previous section.

\subsection{The map $\Psi$ in (\ref{comdia1}).}
We list some properties of $\md$ as the following proposition, the proof of which can be found in \cite{LP1}, \cite{Da2}, \cite{Yuan1}, and \cite{Yuan2}.
\begin{prop}\label{promd}\begin{enumerate}
\item $\md$ is a good quotient of a smooth quasi-projective variety, hence it is normal and Cohen-Macaulay.  $\md$ is irreducible (Theorem 3.1 in \cite{LP1}). 
\item There is a line bundle $\z_r$ over $\md$ (the determinant line bundle associated to $[\mo_{\p^2}]$ on $\md$), such that $dim~H^0(\z_r)=1$.  For $r=1,2$, $\z_r\cong\mo_{\md}$.  For $r\geq3$, the line bundle $\z_r$ admits a unique divisor $D_{\z_r}$ which consists of sheaves with non trivial global sections.(see \cite{Da2} or Theorem 4.3.1 in \cite{Yuan1})
\item Let $\mf$ be a 1-dimensional sheaf with determinant $rH~(r>0)$ and $\chi(\mf)=0$ on $\p^2$.  If $H^0(\mf)=0$, then $\mf$ is semistable and lies in the following sequence 
\begin{equation}\label{1dim1}
0\ra \mo_{\p^2}(-2)^{\oplus r}\ra\mo_{\p^2}(-1)^{\oplus r}\ra \mf\ra 0.
\end{equation}
\end{enumerate}
\end{prop} 
\begin{proof}We only prove the statement (3).  Since $H^0(\mf)=0$, $\mf$ contains no subsheaf of dimension 0.  By Lemma 2.2 in \cite{Yuan2}, every 1-dimensional pure sheaf $\mf$ lies in a sequence 
\[0\ra\me_{\mf}\otimes\mo_{\p^2}(-1)\ra\me_{\mf}\ra\mf\ra0,\]
where $\me_{\mf}$ is a direct sum of line bundles.  Write $\me_{\mf}=\oplus_{i=1}^k\mo_{\p^2}(a_i)$ with $a_1\leq \cdots\leq a_k$.  If $\mf$ is of determinant $rH$ and Euler characteristic 0, then $k=r$ and $\Sigma_{i=1}^k a_k=-r$.  If moreover $H^0(\mf)=0$ then for every subsheaf $\mf'\subset\mf$, $H^0(\mf')=0$ and $\chi(\mf')=h^0(\mf')-h^1(\mf')\leq 0$.  Hence $\mf$ is semistable and $\Sigma_{i=1}^r a_i=-r$ with $a_i\leq -1$ for all $1\leq i\leq r$.  Therefore $a_1=\cdots=a_r=-1$ and $\me_{\mf}\cong\mo_{\p^2}(-1)^{\oplus r}$.
\end{proof}

Let $U(rH,0):=\md\setminus D_{\z_r}$.  Then by Proposition \ref{promd} (3) we have a map $\Psi:U(rH,0)\ra M(Q,(r,r))$.  Easy to see that $\Psi$ is injective and $\Psi(\mf)$ is stable iff $\mf$ is.  On the other hand, a point $[V]\in M(Q,(r,r))$ which can be represented by a representation $V$ of $Q$ lies in the image of $\Psi$ if and only if $\det(x\cdot V(x)+y\cdot V(y)+z\cdot V(z))\neq 0$, in other words the map $\mo_{\p^2}(-2)^{\oplus r}\xrightarrow{x\cdot V(x)+y\cdot V(y)+z\cdot V(z)}\mo_{\p^2}(-1)^{\oplus r}$ induced by $V$ is injective.  Still a priori we don't know whether $\Psi$ is dominant. 

\subsection{Fourier transform on $\p^2$ and the map $\Phi$ in (\ref{comdia1}).}\label{FT} 
We recall the Fourier transform on $\p^2$ (see also Section 4 in \cite{LP1} or Section 3 in \cite{Yuan5}).  Let $\mathcal{D}$ be the universal curve in $\p^2\times |H|$ as follows.
\begin{equation}\label{original}\xymatrix@C=0.01cm{
  \p^2\times |H|\ar[d]^{\tilde{p}}\ar[rd]^{\tilde{q}}&~~\mathcal{D}\ar@{_{(}->}[l]\ar[d]^q\ar[r]^p
                & \p^2 \\
            \p^2   &~~~~~~ |H|\cong\p^2&
               }.
\end{equation}

Let $\mf$ be a pure 1-dimensional sheaf of class $u_d$, then its Fourier transform is defined to be $\mg_{\mf}:=q_{*}(p^{*}(\mf\otimes\mo_{\p^2}(2)))\otimes\mo_{|H|}(-1)$.  Let $\mg$ be a torsion-free sheaf on $|H|$ of class $c_n^r$, then its Fourier transform is defined to be $\mf_{\mg}:=R^1p_{*}(q^{*}(\mg\otimes\mo_{|H|}(-1)))\otimes\mo_{\p^2}(-1)$.  We can identify $|H|$ with $\p^2$.  Then although these two Fourier transforms in general need not be the inverse to each other, they provide a birational map as follows.
\begin{equation}\label{bcor}\Phi:\md \dashrightarrow M(r,0,r).\end{equation}

By Lemma 4.2 and Corollary 4.3 in \cite{LP1}, $\Phi$ is well-defined over $U(rH,0)$ and induces an isomorphism to its image $V(r,0,r):=\Phi(U(rH,0))$.  Since both $\md$ and $M(r,0,r)$ are normal and irreducible, by Zariski's main theorem $\Phi$ can be well-defined outside a subset of codimension at least 2.  By Lemma A.2 and Lemma A.3 in \cite{Yuan5}, the Fourier transform $\mg_{\mf}$ is semistable if $\mf$ is in the following subset $\widetilde{U}(rH,0)$ with complement of codimension $\geq 2$ in $\md$
$$\widetilde{U}(rH,0):=\big\{\mf\in\md\big|\begin{array}{c}Supp(\mf) \text{ is integral, and } \\ h^0(\mf)=h^1(\mf)\leq 1.\end{array}\big\}.$$

For every pure 1-dimensional sheaf $\mf$ of class $u_d$, define its \emph{D-dual} $\mf^D:=\E xt^1(\mf,K_{\p^2})$.  Then $\mf^D$ is also a 1-dimensional pure sheaf of class $u_d$.  By Corollary A.5 in \cite{Yuan5}, we have an isomorphism $\kappa:\mdh\ra\mdh$ by sending each $\mf$ to $\mf^D$.

\begin{lemma}\label{FTcom}Let $\mg$ be a torsion-free sheaf of class $c^r_n~(r\leq n)$ such that $\mg|_{\ell}\cong\mo_{\ell}^{\oplus r}$ for a generic line $\ell\in |\mo_{|H|}(1)|$.  Then its Fourier transform $\mf_{\mg}$ is purely 1-dimensional of class $u_n$.  If moreover $\mg$ is locally free with $\mg^{\vee}$ its dual, then $\mf_{\mg^{\vee}}\cong \mf_{\mg}^{D}.$ 
\end{lemma}
\begin{proof}At first let $n=r$.  By definition $\mf_{\mg}:=R^1p_{*}(q^{*}(\mg\otimes\mo_{|H|}(-1)))\otimes\mo_{\p^2}(-1)$.  Since $\mg|_{\ell}\cong\mo_{\ell}^{\oplus r}$ for a generic line $\ell\in |\mo_{|H|}(1)|$, we know that $\mg$ is $\mu$-semistable and $p_{*}(q^{*}(\mg\otimes\mo_{|H|}(-1)))=0$.  We take a locally free resolution of $\mg$ as follows.
\begin{equation}\label{resg1}0\ra\mk\ra\mo_{\p^2}(-m)^{\oplus h^0(\mg(m))}\ra \mg\ra0,
\end{equation} 
with $m\gg0$.  Then we have 
\begin{equation}\label{resg2}0\ra R^1p_{*}(q^{*}(\mk\otimes\mo_{|H|}(-1)))\ra R^1p_{*}(q^{*}\mo_{|H|}(-m-1)^{\oplus h^0(\mg(m))})\ra \mf_{\mg}\otimes\mo_{\p^2}(1)\ra0,
\end{equation} 
where $ R^1p_{*}(q^{*}(\mk\otimes\mo_{|H|}(-1)))$ and $R^1p_{*}(q^{*}\mo_{|H|}(-m-1)^{\oplus h^0(\mg(m))})$ are locally free.  Hence $\mf_{\mg}$ is of homological dimension 1 and hence pure.  The class of $\mf_{\mg}$ in $K(\p^2)$ only depends on the class of $\mg$ in $K(|H|)$.  

If $n>r$, then $\mg\oplus\mo_{|H|}^{\oplus n-r}$ is torsion free of class $c_n^n$ and $\mf_{\mg}\cong\mf_{\mg\oplus\mo_{|H|}^{\oplus n-r}}$ since $p_{*}(q^{*}\mo_{|H|}(-1))=R^1p_{*}(q^{*}\mo_{|H|}(-1))=0$.
 
By Grothendieck duality (or Lemma 5.5 in \cite{Abe}), for $\mg$ locally free we have 
$$R^1p_{*}(\hh om(q^*(\mg\otimes\mo_{|H|}(-1)), \omega_{\mathcal{D}/\p^2}))\cong\E xt^1(R^1p_*(q^*(\mg\otimes\mo_{|H|}(-1))),\mo_{\p^2}),$$ where $\omega_{\mathcal{D}/\p^2}$ is the relative dualizing sheaf of the map $p$.  Since $\omega_{\mathcal{D}/\p^2}\cong p^*{\mo_{\p^2}(1)}\otimes q^*\mo_{|H|}(-2)$, we have 
\scriptsize
$$\mf_{\mg^{\vee}}=R^1p_{*}(q^*(\mg^{\vee}\otimes\mo_{|H|}(-1)))\otimes\mo_{\p^2}(-1)\cong\E xt^1(R^1p_*(q^*(\mg\otimes\mo_{|H|}(-1)))\otimes\mo_{\p^2}(-1),K_{\p^2})=\mf_{\mg}^D.$$
\normalsize\end{proof}

Denote by $M(r,0,n)^{b}$ the subset of $M(r,0,n)$ consisting of locally free sheaves.  It is easy to find that $M(r,0,n)\setminus M(r,0,n)^b$ is of codimension $\geq r-1$ in $M(r,0,n)$ (see Proposition 2.8 in \cite{DLP}).  There is a birational map $\zeta:M(r,0,n)\dashrightarrow M(r,0,n)$ sending each $\mu$-stable bundle to its dual.  Since $M(r,0,n)$ is normal, by Zariski's main theorem $\zeta$ can be well-defined outside a subset of codimension $\geq2$.  If $r=2$, then $\zeta$ is just the identity.  If $n>r\geq 3$, by Lemma 2.10 in \cite{Yuan7} strictly $\mu$-semistable sheaves form a closed subset of codimension $\geq2$.  

However if $n=r\geq 3$, then by Proposition 3.1 in \cite{Yuan7} there is a divisor $\textbf{S}_r$ consisting of strictly $\mu$-semistable sheaves.  For every $\mg\in\textbf{S}_r\cap \wrr$, the dual bundle $\mg^{\vee}$ can not be semistable since $H^0(\mg^{\vee})\cong \Hom(\mg,\mo_{\p^2})\neq0$.  Define $V(r,0,r)^b:=V(r,0,r)\cap M(r,0,r)^b$, then $M(r,0,r)\setminus V(r,0,r)^b$ is of codimension $\geq 2$ by Lemma A.3 in \cite{Yuan5}.  The following lemma is a direct consequence of Lemma \ref{FTcom}.
\begin{lemma}\label{pfmain2} $\zeta$ can be well-defined over $V(r,0,r)^b$ and we have the following commutative diagram
\[\xymatrix{V(r,0,r)^b\ar[r]^{\zeta}_{\cong}& V(r,0,r)^b\\ U(rH,0)^b\ar[u]^{\Phi}_{\cong}\ar[r]^{\kappa}_{\cong}&U(rH,0)^b\ar[u]_{\Phi}^{\cong}},\]
where $U(rH,0)^b:=\Phi^{-1}(V(r,0,r)^b)$.
\end{lemma}
\begin{rem}Let $\mg$ be a stable bundle in $\textbf{S}_r$ lying in the following exact sequence 
\[0\ra\mg_1\ra\mg\ra\mo_{\p^2}\ra0.\]
Then $\zeta(\mg)$ lies in the following exact sequence
\[0\ra\mg_1^{\vee}\ra\zeta(\mg)\ra\mo_{\p^2}\ra 0.\]
\end{rem}

\subsection{Commutativity of the diagram in (\ref{comdia1}).}
\begin{prop}\label{comprop}The diagram in (\ref{comdia1}) commutes.
\end{prop}
\begin{proof}
We can restrict ourselves to $U(rH,0)$ where both $\Psi$ and $\Phi$ in (\ref{comdia1}) are well-defined.  So we want to show the following diagram commutes
\begin{equation}\label{comdia2}\xymatrix@C=1cm{U(rH,0)\ar[r]^{\Psi\quad}\ar[d]_{\Phi}& M(Q,(r,r))\ar[d]^{g}_{\cong}\\ M(r,0,r)\ar[r]_{f\quad}^{\cong\quad}&M(Q,(r,2r))}.
\end{equation}
For any $\mf\in U(rH,0)$ we have the following exact sequence as in (\ref{1dim1})
\begin{equation}\label{1dim2}0\ra \mo_{\p^2}(-2)^{\oplus r}\xrightarrow{x\cdot A_x+y\cdot A_y+z\cdot A_z}\mo_{\p^2}(-1)^{\oplus r}\ra \mf\ra 0.
\end{equation}
So $\Psi(\mf)=[\xymatrix@C=1cm{\mathbb{C}^r\ar[r]^{A_y}\ar@/^1pc/[r]^{A_x}\ar@/_1pc/[r]^{A_z}&\mathbb{C}^r}]$.  

On $\p^2\times |H|$, we have 
\begin{equation}\label{trans1}0\ra \tilde{p}^*\mo_{\p^2}(-1)\otimes \tilde{q}^*\mo_{|H|}(-1)\ra\mo_{\p^2\times |H|}\ra\mo_{\mathcal{D}}\ra0.
\end{equation}
Do the Fourier transform to (\ref{1dim2}) and we have
\begin{equation}\label{trans2}\xymatrix@C=0.3cm@R=0.9cm{&&0\ar[d]&0\ar[d]&\\&0\ar[r]\ar[d]&H^0(\mo_{\p^2})^{\oplus r}\otimes\mo_{|H|}(-2)\ar[r]^{\qquad\cong}\ar[d]_{F_2}&\mo_{|H|}(-2)^{\oplus r}\ar[r]\ar[d]_{x^*\cdot \widetilde{A}_x+y^*\cdot\widetilde{A}_y+z^*\cdot \widetilde{A}_z}&0\\
0\ar[r]& H^0(\mo_{\p^2})^{\oplus r}\otimes\mo_{|H|}(-1)\ar[r]^{F_1~}\ar[d]^{\cong} & H^0(\mo_{\p^2}(1))^{\oplus r}\otimes \mo_{|H|}(-1)\ar[r]\ar[d]&\mo_{|H|}(-1)^{\oplus 2r}\ar[r]\ar[d] &0\\
 0\ar[r] & q_{*}(p^*\mo_{\p^2}^{\oplus r})\otimes\mo_{|H|}(-1)\ar[r]\ar[d] & q_{*}(p^*\mo_{\p^2}(1)^{\oplus r})\otimes\mo_{|H|}(-1)\ar[r]\ar[d]&\mg_{\mf}\ar[r]\ar[d]&0\\&0&0&0&}
\end{equation}
where $[x^*,y^*,z^*]$ are the homogenous coordinates on $|H|$.  $f(\mg_{\mf})=f\circ\Phi(\mf)$ can be represented by three $r\times 2r$ matrices $(\widetilde{A}_x,\widetilde{A}_y,\widetilde{A}_z)$.  The map $F_1:\mo_{|H|}(-1)^{\oplus r}\ra\mo_{|H|}(-1)^{\oplus 3r}$ in (\ref{trans2}) is given by the matrix $(A_x,A_y,A_z)\otimes id_{\mo_{|H|}(-1)}$.  The map $F_2:\mo_{|H|}(-2)^{\oplus r}\ra\mo_{|H|}(-1)^{\oplus 3r}$ in (\ref{trans2}) is given by the matrix $(x^*\mathbf{I}_r,y^*\bi_r,x^*\bi_r)$.  On the other hand, by the commutativity of (\ref{trans2}), the map 
$$\left(\begin{array}{c}F_1\\ \oplus \\ F_2\end{array}\right):\begin{array}{c}\mo_{|H|}(-1)^{\oplus r}\\ \oplus\\ \mo_{|H|}(-2)^{\oplus r}\end{array}\ra\mo_{|H|}(-1)^{\oplus 3r}$$
can also be represented by the matrix $\begin{pmatrix}\bi_r,&\mathbf{0}_{r\times 2r}\\ *, & x^*\cdot \widetilde{A}_x+y^*\cdot\widetilde{A}_y+z^*\cdot \widetilde{A}_z\end{pmatrix}$.  Hence we know that $\exists~ P\in GL(3r,\mathbb{C})$ such that
\begin{equation}\label{compare2}\begin{pmatrix}A_x,&A_y, &A_z\\ x^*\bi_r, & y^*\bi_r,&z^*\bi_r\end{pmatrix}\cdot P=\begin{pmatrix}\bi_r,&\mathbf{0}_{r\times 2r}\\ *, & x^*\cdot \widetilde{A}_x+y^*\cdot\widetilde{A}_y+z^*\cdot \widetilde{A}_z\end{pmatrix}.\end{equation}
Compare (\ref{compare2}) with (\ref{compare1}) and we see the commutativity of (\ref{comdia2}).  Hence the proposition.   
\end{proof}

\begin{rem}\label{extend1}For $\mf\in\widetilde{U}(rH,0)$ such that $h^0(\mf)=1$, by Lemma A.2 in \cite{Yuan5}, its Fourier transform $\mg_{\mf}$ is strictly semistable and $S$-equivalent to $S^2\mt_{\p^2}(-1)\oplus\mg_{\mf'}$, where $\mt_{\p^2}$ is the tangent bundle of $\p^2$ and $\mf'\in U((r-3)H,0)$ uniquely determined by $\mf$.  Hence by the Proposition \ref{comprop}, $\Psi(\mf)=\Lambda_3\oplus \Psi(\mf')$ where $\Lambda_3=g^{-1}\circ f(S^2\mt_{\p^2}(-1))\in M(Q,(3,3))$ and it can be represented by matrices $(A^{\Lambda}_{x},A^{\Lambda}_{y},A^{\Lambda}_{z})$ such that $x\cdot A^{\Lambda}_{x}+y\cdot A^{\Lambda}_{y}+z\cdot A^{\Lambda}_{z} =\begin{pmatrix}y,&-z,&0\\-x,&0,&z\\0,&x,&-y\end{pmatrix}$.\end{rem} 

\section{Strange duality on $\p^2$.}
\subsection{The problem.}
We have two moduli spaces $\mdh~(d>0)$ and $M(r,0,n)~(n\geq r>0)$ parametrizing semistable sheaves of class $u_d$ and $c^r_n$ respectively.  We have the so-called \emph{determinant line bundle} $\lcn$ ($\lambda_{c_n^r}(d)$, resp.) over $\mdh$ ($M(r,0,n)$, resp.) associated to $c^r_n$ ($u_d$, resp.).  We have a \emph{strange duality map} well-defined up to scalars as follows.
\begin{equation}\label{sdmap}SD_{c_n^r,d}:H^0(M(r,0,n),\lambda_{c^r_n}(d))^{\vee}\ra H^0(\mdh,\lcn).
\end{equation}

The map $SD_{c_n^r,d}$ in (\ref{sdmap}) is induced by the the section $\sigma_{\crn,d}$ of the line bundle $\lcd\boxtimes\lcn$ over 
$M(r,0,n)\times\mdh$, whose zero set is 
\begin{equation}\label{sddiv}\mathtt{D}_{\crn,d}:=\big\{(\mg,\mf)\in   M(r,0,n)\times \mdh \bigm| h^0(\p^2,\mg\otimes \mf)=h^1(\p^2,\mg\otimes\mf)\ne 0\big\}.\end{equation}

The strange duality conjecture on $\p^2$ due to Le Potier (Conjecture 2.2 in \cite{Da2}) is as follows.  
\begin{con/que}Is $SD_{\crn,d}$ an isomorphism? 
\end{con/que}
For details of the setting up including the explicit definition of the determinant line bundles, we omit here and refer to  \S1 and \S2 in \cite{Da2}, or \S2 and \S3 in \cite{Yuan1}, or \S2.4 in \cite{GY}, or \S2.3 in \cite{Yuan6}.  For more properties of the determinant line bundle, we refer to Chapter 8 in \cite{HL} noting that the definition in \cite{HL} is dual to us. 

Recall in Proposition \ref{promd} (2), we have introduced the line bundle $\z_d$ over $\mdh$ which is the determinant line bundle associated to the class $[\mo_{\p^2}]\in K(\p^2)$ and $H^0(\z_d)=1$.  Denote by $\theta_d$ the unique non-zero section up to scalars, which vanishes at points corresponding to sheaves with non-trivial global sections.  By the basic property of the determinant line bundles (see e.g. \S2.1 in \cite{GY}, or \S3 in \cite{Yuan1}), we have $\lambda_{d}(c^r_n)\otimes\z_d^{\otimes (n-r)}\cong\lambda_{d}(c^n_n)$ for all $n\geq r$.  We then have the inclusion map
\begin{equation}\label{mzeta}\jmath^r_n:H^0(\mdh,\lambda_{d}(c^r_n))\xrightarrow{.\theta_d^{n-r}}H^0(\mdh,\lambda_d(c^n_n)),
\end{equation}  
defined by multiplying $n-r$ times of the section $\theta_d$. 
\begin{rem}\label{atform}Actually by Proposition 2.8 in \cite{LP2}, we have $\lambda_{d}(c_n^r)\cong\z_d^{\otimes r}\otimes \pi^{*}\mo_{|dH|}(n)$ for all $r,n$, where $\pi:\mdh\ra|dH|$ sends every sheaf to its support.\end{rem}

\begin{rem}\label{dlcom}Recall we have two morphisms: $\kappa:\mdh\ra\mdh$ sending each $\mf$ to $\mf^D$, and $\zeta:V(r,0,r)^b\ra V(r,0,r)^b$ as defined in Lemma \ref{pfmain2}.  By Corollary A.5 and Corollary 3.7, we have $\kappa^*\lambda_d(c^r_n)\cong\lambda_d(c^r_n)$ and $\zeta^*\lambda_{c^r_r}(d)\cong\lambda_{c^r_r}(d)$.
\end{rem}

\subsection{The results.} 
Our main results are the following two theorems.
\begin{thm}\label{mainthm1}The strange duality map $SD_{c_n^r,d}$ in (\ref{sdmap}) is an isomorphism for $r=n>0$ and $d>0$.
\end{thm}
\begin{thm}\label{mainthm2}The strange duality map $SD_{c_n^r,d}$ in (\ref{sdmap}) is injective for all $n\geq r>0$ and $d>0$.
\end{thm}
Together with Proposition 4.1 and Theorem 4.16 (1) in \cite{Yuan7}, Theorem \ref{mainthm1} imply the following corollary directly.
\begin{coro}\label{maincoro1}The strange duality map $SD_{c_n^r,d}$ in (\ref{sdmap}) is an isomorphism for $r>0$, $n=r+1$ and $d>0$.
\end{coro}
Together with Proposition 4.14 in \cite{Yuan7}, Theorem \ref{mainthm2} imply the following corollary directly.
\begin{coro}\label{maincoro2}The strange duality map $SD_{c_n^r,d}$ in (\ref{sdmap}) is an isomorphism for $n\geq r>0$ and $d=1,2,3$.
\end{coro}

\subsection{A variation of the strange duality map.}
In order to relate the map $SD_{c^r_n,d}$ in (\ref{sdmap}) to the map $SD(Q)$ in (\ref{sdquiver}), we define a variation of the strange duality map $SD_{c,u}$ in the sheaf theory, which we will denote by $VD_{c,u}$.

Let $Y$ be a projective smooth scheme of dimension $\dim(Y)$ with canonical line bundle $K_Y$.  For every element $u$ in the Grothendieck group $K(Y)$ of coherent sheaves, we can find finitely many bundles $\me_1,\cdots,\me_n$ such that $u=\displaystyle{\sum_{i=1}^n }k_i[\me_i]$ with $[\me_i]$ the class of $\me_i$ in $K(Y)$.  We define $u^{\vee}:=\displaystyle{\sum_{i=1}^n }k_i[\me^{\vee}]$ where $\me^{\vee}:=\mathscr{H}om(\me_i,\mo_Y)$, and $u^D:=
u^{\vee}\otimes K_Y$. 
Easy to see $(u^D)^D=u$ and $(u^{\vee})^{\vee}=u$.  By Serre duality $\chi(c\otimes u)=(-1)^m\chi(c^{\vee}\otimes u^{D})$ for every $c,u\in K(Y)$.  

Let $S$ be a Noetherian scheme.  $Q:S\times Y\ra Y$ and $p:S\times Y\ra S$ are the projections. 
Let $\E$ be a sheaf over $S\times Y$ which is a $S$-flat family of coherent sheaves on $Y$.  Then analogous to the determinant line bundle map (Definition 8.1.1 in \cite{HL}), we have the following two well-defined morphisms

\begin{defn}\label{vdlbl}Let $\nu^l_{\E}: K(Y)\ra\Pic(S)$ be the composition of the homomorphisms:
\begin{equation}\label{vdlbf}\xymatrix@C=1.2cm{K(Y)=K^0(Y)\ar[r]^{q^*}&K^0(Y\times S)\ar[d]_{R^{\bullet}\hh om(\E,-)}&&\\ &K^0(Y\times S)
\ar[r]^{R^{\bullet}p_*}&K^0(S)\ar[r]^{\det^{-1}}&\Pic(S),}\end{equation}
where $K^0(-)$ is the subgroup of $K(-)$ generated by classes of locally free sheaves.

Let $\nu^h_{\E}: K(Y)\ra\Pic(S)$ be the composition of the homomorphisms:
\begin{equation}\label{vdlbfr}\xymatrix@C=1.2cm{K(Y)=K^0(Y)\ar[r]^{q^*}&K^0(Y\times S)\ar[d]_{R^{\bullet}\hh om(-,\E)}&&\\ &K^0(Y\times S)
\ar[r]^{R^{\bullet}p_*}&K^0(S)\ar[r]^{\det^{-1}}&\Pic(S),}\end{equation}
\end{defn}
 
Recall that the determinant line bundle map $\lambda_{\E}: K(Y)\ra\Pic(S)$ (dual to Definition 8.1.1 in \cite{HL}) is the composition of the homomorphisms:
\begin{equation}\label{dlbf}\xymatrix@C=1.2cm{K(Y)=K^0(Y)\ar[r]^{q^*}&K^0(Y\times S)\ar[d]_{\overset{L}{\otimes}~\E}&&\\ &K^0(Y\times S)
\ar[r]^{R^{\bullet}p_*}&K^0(S)\ar[r]^{\det^{-1}}&\Pic(S),}\end{equation}
where $\overset{L}{\otimes}$ is the flat tensor, i.e. $[\F\overset{L}{\otimes}\E]:=\displaystyle{\sum_{i\geq 0}}(-1)^i[\Tor^i(\F,\E)]$.

\begin{rem}\label{vtodl}If $R^{i}\hh om(\E,\mo_{S\times Y})=0$ for all $i\neq i_0$, then by definition for all $u\in K(Y)$ we have
$$\lambda_{R^{i_0}\hh om(\E,\mo_{S\times Y})}(u)^{\otimes (-1)^{i_0}}\cong\nu^l_{\E}(u).$$ 
Moreover by Grothendieck duality (or Lemma 5.5 in \cite{Abe}), we have 
for all $w\in K^0(Y\times S)$
\[\det~^{\otimes (-1)^{\dim(Y)+1}}[R^{\bullet}p_*(q^*K_Y~\otimes~ w^{\vee})]\cong \det [R^{\bullet}p_*w].\]
Hence we have
\begin{equation}\label{vtodlf1}\lambda_{\E}(u)^{\otimes (-1)^{\dim(Y)+1+i_0}}\cong\lambda_{R^{i_0}\hh om(\E,q^*K_{ Y})}(u^{\vee});
\end{equation}
and
\begin{equation}\label{vtodlf2}\lambda_{\E}(-u^D)^{\otimes (-1)^{\dim(Y)}}\cong\lambda_{R^{i_0}\hh om(\E,\mo_{S\times Y})}(u)^{\otimes (-1)^{i_0}}\cong\nu^l_{\E}(u).
\end{equation}
\end{rem}
\begin{rem}\label{vtodr}
By definition for all $u\in K(Y)$ we have
$$\lambda_{\E}(u^{\vee})\cong\nu^h_{\E}(u).$$ 
\end{rem}

Analogously we can get well-defined line bundles $\nu_u^l(c)$ and $\nu_u^h(c)$ over the moduli spaces $M_Y(u)$ of semistable sheaves of class $u$ for $c$ in some subgroups of $K(Y)$, even though there is no universal family over $Y\times M_Y(u)$.  The maps $\nu^l_u,\nu^h_u$ also satisfy some kind of universal properties analogous to Theorem 8.1.5 in \cite{HL}.  For simplicity we now restrict ourselves to $Y=\p^2$.  Then $\nu_u^l(c)$ ($\nu_u^h(c)$, resp.) is well-defined if $\chi(u,c)=0$ ($\chi(c,u)=0$, resp.).

\begin{lemma}\label{VSD1}Let $S,T$ be two schemes of finite type.  Let $c,u\in K(\p^2)$ such that $\chi(c,u)=0$.  Let $\F$ ( $\G$, resp.) be a flat family of sheaves of class $u$ ( $c$, resp.) over $\p^2$, parametrized by $S$ ( $T$, resp.).  Assume moreover the following two conditions:
\begin{enumerate}
\item $\E xt^i(\G_t,\F_s)=0$, $\forall s\in S,~t\in T,~i>0$;
\item $H^2(\hh om(\G_t,\F_s))=0$, $\forall s\in S, ~t\in T$.
\end{enumerate}  
Then there is a canonical section $\varsigma_{\F,\G}\in H^0(S\times T,\nu^h_{\F}(c)\boxtimes\nu^l_{\G}(u))$, unique up to scalars, whose zero set is 
$$\mathtt{B}_{\F,\G}:=\big\{(s,t)\in  S\times T \bigm| \emph{hom}(\G_t,\F_s)=\emph{ext}^1(\G_t,\F_s)\ne 0\big\}.$$
\end{lemma} 
\begin{proof}We have three projections: $p_{S\times T}:\p^2\times S\times T\ra S\times T$, $p_S:S\times T\ra S$ and $p_T:S\times T\ra T$.  By the condition (1), we have $\E xt^i((id_{\p^2}\times p_T)^*\G,(id_{\p^2}\times p_S)^*\F)=0$ for all $i>0$,  hence $\hh om((id_{\p^2}\times p_T)^*\G,(id_{\p^2}\times p_S)^*\F)$ is flat over $S\times T$.  By Proposition 2.1.10 in \cite{HL}, there is a locally free resolution
\[0\ra F_2\ra F_1\ra F_0\ra \hh om((id_{\p^2}\times p_T)^*\G,(id_{\p^2}\times p_S)^*\F)\ra0\]
such that $R^j(p_{S\times T})_*F_i=0$ and $R^2(p_{S\times T})_*F_i$ is locally free for $i=0,1,2$ and $j=0,1$.  Moreover we have a complex
\begin{equation}\label{complex}F^{\bullet}:=[R^2(p_{S\times T})_*F_2\ra R^2(p_{S\times T})_*F_1\xrightarrow{r_0} R^2(p_{S\times T})_*F_0]\end{equation}
with $h^i(F^{\bullet})=R^{2-i}(p_{S\times T})_*\hh om((id_{\p^2}\times p_T)^*\G,(id_{\p^2}\times p_S)^*\F).$

By the condition (2), we have $R^2(p_{S\times T})_*\hh om((id_{\p^2}\times p_T)^*\G,(id_{\p^2}\times p_S)^*\F)=0$ and hence the map $r_0$ in (\ref{complex}) is surjective.  Therefore we have a two-term complex of locally free sheaves on $S\times T$ as follows
\begin{equation}\label{complex1}F_1^{\bullet}:=[R^2(p_{S\times T})_*F_2\xrightarrow{r_1} \ker(r_0).]\end{equation}
By the see-saw lemma (c.f. Lemma 2.10 in \cite{Da2}), it is easy to see $\varsigma_{\F,\G}:=\det(r_1)$ gives a section of the line bundle $\nu^h_{\F}(c)\boxtimes\nu^l_{\G}(u)$ whose zero set is 
$$\mathtt{B}_{\F,\G}:=\big\{(s,t)\in  S\times T \bigm| \text{hom}(\G_t,\F_s)=\text{ext}^1(\G_t,\F_s)\ne 0\big\}.$$
\end{proof}

Recall that $M(r,0,n)^{b}$ is the subset of $M(r,0,n)$ consisting of locally free sheaves.  We have the following proposition which is a $(\nu^l,\nu^h)$-analog to Theorem 2.1 in \cite{Da2} or Proposition 2.5 in \cite{Yuan6}.
\begin{prop}\label{vglob}(1) There is a canonical section $\varsigma_{\crn,u_d}\in H^0(M(r,0,n)^b\times\mdh,\nu^l_{c^r_n}(u_d)\boxtimes\nu^h_{d}(\crn))$, unique up to scalars, whose zero set is 
\begin{equation}\label{vsddiv}\mathtt{B}_{\crn,u_d}:=\big\{([\mg],[\mf])\in   \wrn^b\times \mdh\bigm| \emph{hom}(\mg, \mf)=\emph{ext}^1(\mg,\mf)\ne 0\big\}.\end{equation}

(2) The section $\varsigma_{\crn,u_d}$ defines a linear map up to scalars
\begin{equation}
\label{VDmap}
VD_{\crn,u_d}:H^0(\wrn^b,\nu^l_{c^r_n}(u_d))^\vee \to H^0(\mdh,\nu^h_{d}(\crn)).
\end{equation}

(3) Denote by $\varsigma^h_{\mg}$ ($\varsigma^l_{\mf}$, resp.) the restriction of $\varsigma_{\crn,u_d}$ to $\{\mg\}\times\mdh$ ($\wrn^b\times\{\mf\}$, resp.).  Then $\varsigma^h_{\mg}$ ($\varsigma^l_{\mf}$, resp.) only depends (up to scalars) on the S-equivalence class of $\mg$ ($\mf$, resp.).  


\end{prop}

\begin{lemma}\label{vtos}$\nu^l_{c^r_n}(u_d)\cong\lambda_{c^r_n}(u_d)$ and $\nu^h_{d}(c^r_n)\cong\lambda_{d}(c^r_n)$.
\end{lemma} 
\begin{proof}Notice that $-u_d^D=u_d$ and $(c^r_n)^{\vee}=c^r_n$.  The lemma follows straightforward from (\ref{vtodlf2}) in Remark \ref{vtodl} and Remark \ref{vtodr}. 
\end{proof}

Now we have two maps as follows.
\[SD_{\crn,u_d}:H^0(\wrn^b,\lambda_{c^r_n}(u_d))^\vee \to H^0(\mdh,\lambda_{d}(c^r_n));\] 
\[VD_{\crn,u_d}:H^0(\wrn^b,\lambda_{c^r_n}(u_d))^\vee \to H^0(\mdh,\lambda_{d}(c^r_n)).\]
By switching $\crn$ and $u_d$, we also have 
\[SD_{u_d,\crn}:H^0(\mdh,\lambda_{d}(c^r_n))^\vee \to H^0(\wrn^b,\lambda_{c^r_n}(u_d)).\]

It is easy to see that $SD_{u_d,\crn}$ is dual to $SD_{\crn,u_d}$.
 We will see later that $VD_{c^r_r,u_d}$ somehow equals to $SD_{u_d,c^r_r}$ after Fourier transform.  First we have the following proposition.
\begin{prop}\label{VDtoSD}\begin{enumerate}
\item The Fourier transform $\Phi:~M(dH,0)\dashrightarrow M(d,0,d)$ is a birational map of normal projective schemes and $\Phi^{*}\lambda_{c^d_d}(u_r)\cong\Phi^{*}\lambda_{c^d_d}(-u_r^D)
\cong\lambda_d((c^r_r)^{\vee})\cong\lambda_d(c^r_r)$, $\forall~d, r>0$.  

Moreover $\Phi^{*}: H^0(M(d,0,d),\lambda_{c^d_d}(u_r))\xrightarrow{\cong} H^0(\md,\lambda_{d}(c^r_r))$ is an isomorphism.

\item The restriction map $H^0(\mdh,\lambda_{d}(c^r_r))\ra H^0(U(dH,0)^b,\lambda_{d}(c^r_r))$ is an isomorphism, for all $d, r >0$. 

\item 
We have the following commutative diagram
\begin{equation}\label{SDtoVD}\xymatrix@C=2cm{H^0(U(dH,0)^b,\lambda_{d}(c^r_r))^{\vee}\ar[r]^{SD_{u_d,c^r_r}\circ(\kappa^*)^{\vee}}&H^0(V(r,0,r),\lambda_{c^r_r}(u_d))\\
H^0(V(d,0,d)^b,\lambda_{c_d^d}(u_r))^{\vee}\ar[u]^{(\Phi^{*})^{\vee}}_{\cong}\ar[r]_{VD_{c_r^r,u_d}} &H^0(U(rH,0),\lambda_r(c_d^d))\ar[u]_{\Phi^*}^{\cong}}.
\end{equation}
\end{enumerate}
\end{prop}
\begin{proof}The strategy of the proof is the same as Proposition A.9 in \cite{Yuan5}.

We have a map $\Phi:~\widetilde{U}(dH,0)\rightarrow M(d,0,d)$.  Denote by $\mathfrak{M}(dH,0)\xrightarrow{\pi_{u_d}} M(dH,0)$ and $\km(r,0,r)\xrightarrow{\pi_{c^r_r}} M(r,0,r)$ the two good quotients from which one constructs the moduli spaces.  Denote by $\widetilde{\ku}(dH,0),\ku(dH,0),\kv(r,0,r)$ and $\kv(r,0,r)^b$ the preimages of $\widetilde{U}(dH,0),U(dH,0),V(r,0,r)$ and $V(r,0,r)^b$ respectively.

To prove (1), it is enough to show that over $\widetilde{\ku}(dH,0)$ we have $\pi_{u_d}^*\Phi^{*}\lambda_{c^d_d}(-u_r^D)\cong\pi_{u_d}^* \lambda_d((c^r_r)^{\vee})$. Let $\F_{u_d}$ be the universal family over $\p^2\times \widetilde{\ku}(dH,0)$.  Recall in (\ref{original}) we have the universal curve $\mathcal{D}$ in $\p^2\times |H|$ and we have the following commutative diagram.   
\[\xymatrix@C=2cm{&
\widetilde{\ku}(dH,0)\times\mathcal{D}\ar[ld]_{p_1}\ar[d]_{id_{\widetilde{\ku}(dH,0)}\times q}\ar[r]^{id_{\widetilde{\ku}(dH,0)}\times p}
                & \widetilde{\ku}(dH,0)\times\p^2\ar[d]^{p_0}\ar[lld]^{p_2} \\
{\widetilde{\ku}(dH,0)}&\widetilde{\ku}(dH,0)\times |H|\ar[l]_{p_2}\ar[r]^{p_0}&|H|\cong \p^2
               }.\]
Let $\mg\in V(r,0,r)^b$ and let $\mg=\mg_{\mf}=q_{*}(p^{*}(\mf\otimes\mo_{\p^2}(2)))\otimes\mo_{|H|}(-1)$ with $\mf\in U(dH,0)$.  Then define a sheaf $\widetilde{\F}$ over $\widetilde{\ku}(dH,0)\times\mathcal{D}$ as follows. 
$$\widetilde{\F}:=((id_{\widetilde{\ku}(dH,0)}\times p)^*p_0^*(\mf^{D}\otimes\mo_{\p^2}(-1))\otimes ((id_{\widetilde{\ku}(dH,0)}\times q)^*(\F_{u_d}\otimes p_0^*\mo_{\p^2}(2))).$$

We have a locally free resolution for $\mf^D\otimes \mo_{\p^2}(-1)$ as follows.
\begin{equation}\label{rain}0\ra\ma\ra\mb\ra\mf^D\otimes \mo_{\p^2}(-1)\ra0.
\end{equation}
We have for all $j\geq2$
$$\Tor^j(id_{\widetilde{\ku}(dH,0)}\times p)^*p_0^*(\mf^{D}\otimes\mo_{\p^2}(-1)),id_{\widetilde{\ku}(dH,0)}\times q)^*(\F_{u_d}\otimes p_0^*\mo_{\p^2}(2)))=0.$$
Since for a generic $s\in\widetilde{\ku}(dH,0)$ the intersection of the supports of $\mf$ and $(\F_{u_d})_s$ is of dimension 0, we have
$$\Tor^1(id_{\widetilde{\ku}(dH,0)}\times p)^*p_0^*(\mf^{D}\otimes\mo_{\p^2}(-1)),id_{\widetilde{\ku}(dH,0)}\times q)^*(\F_{u_d}\otimes p_0^*\mo_{\p^2}(2)))$$
is a subsheaf of lower dimension of $(id_{\widetilde{\ku}(dH,0)}\times p)^*p_0^*\ma\otimes ((id_{\widetilde{\ku}(dH,0)}\times q)^*(\F_{u_d}\otimes p_0^*\mo_{\p^2}(2)))$, hence it has to be zero.  
Then by Lemma 3.4 and Lemma A.7 (1) in \cite{Yuan5}, we have over $\widetilde{\ku}(dH,0)\times \p^2$
\[R^i(id_{\widetilde{\ku}(dH,0)}\times p)_{*}\widetilde{\F}=0,~\forall~i>0,~(id_{\widetilde{\ku}(dH,0)}\times p)_{*}\widetilde{\F}\cong p_0^{*}\mf^D\otimes\G_{\F_{u_d}},\]
and 
\[R^i(id_{\widetilde{\ku}(dH,0)}\times q)_{*}\widetilde{\F}=0,~\forall~i>0,~(id_{\widetilde{\ku}(dH,0)}\times q)_{*}\widetilde{\F}\cong p_0^{*}\mg^{\vee}\otimes\F_{u_d},\]
where $\G_{\F_{u_d}}$ is the fiber-wise Fourier transform of $\F_{u_d}$, in other words, $\G_{\F_{u_d}}$ induces the classifying map $\widetilde{\ku}(dH,0)\xrightarrow{\Phi\circ\pi_{u_d}}M(d,0,d)$.  Therefore
\[\pi_{u_d}^*\Phi^{*}\lambda_{c^d_d}(-u_r^D)\cong\det~^{-1}(R^{\bullet}(p_1)_{*}\widetilde{\F})\cong\pi_{u_d}^* \lambda_d((c^r_r)^{\vee}).\]
Hence we proved (1).
 
By Lemma A.3 in \cite{Yuan5}, we have that the restriction map $$H^0(M(d,0,d),\lambda_{c_d^d}(u_r))\xrightarrow{restr.} H^0(V(d,0,d),\lambda_{c_d^d}(u_r))$$ is an isomorphism, for all $d, r >0$.  On the other hand by (1), we have the isomorphism $H^0(U(dH,0),\lambda_{d}(c^r_r))\xrightarrow[\cong]{\Phi^*} H^0(V(d,0,d),\lambda_{c_d^d}(u_r))$.  Hence (2) is proved.

To prove (3), we restrict ourselves to $ \kv(r,0,r)^b\times \ku(dH,0)$.  Let $S=\kv(r,0,r)^b$ and $T=\ku(dH,0)$.  Let $\G_S$ ($\F_T$, resp.) be a $S$-flat ($T$-flat, resp.) family of sheaves in $V(r,0,r)^b$ ($U(dH,0)$, resp.) over $\p^2\times S$ ($\p^2\times T$, resp.).
Let $\F_S$ ($\G_T$, resp.) be the fiber-wise Fourier transform of $\G_S$ ($\F_T$, resp.).  Let $\F_S^D=(id_{\p^2}\times \kappa)^*\F_S$ be the fiber-wise D-dual of $\F_S$.  Let $\G_S^{\vee}$ be the fiber-wise dual of $\G_S$ which is a $S$-flat family.  We have the following commutative diagram.
\[\xymatrix@C=1.2cm{S\times \cd\ar @<0.5ex> [d]^{id_S\times q}  \ar @<-0.5ex> [d]_{id_S\times p}& S\times T\times\cd \ar[l]_{\alpha_{S}^{\cd}}\ar @<0.5ex> [d]^{id_{S\times T}\times q}  \ar @<-0.5ex> [d]_{id_{S\times T}\times p}\ar[r]^{\alpha_T^{\cd}} & T\times\cd \ar @<0.5ex> [d]^{id_T\times q}  \ar @<-0.5ex> [d]_{id_T\times p}\\ S\times\p^2\ar[d]_{\tau_S}& S\times T\times \p^2\ar[l]^{\alpha_S}\ar[ld]^{\tau_{S\times T}}\ar[r]_{\alpha_T}\ar[d]_{h}\ar[rd]_{\tau_{S\times T}} & T\times\p^2 \ar[d]^{\tau_T}\\ \p^2 & S\times T&\p^2.}\]

Define
$$\widetilde{\F_{S,T}}:=((id_{S\times T}\times p)^*\alpha_S^*(\F_S^{D}\otimes \tau_S^*\mo_{\p^2}(-1)))\otimes ((id_{S\times T}\times q)^*\alpha_T^*(\F_T\otimes \tau_{T}^*\mo_{\p^2}(2))).$$
Since for a generic $(s,t)\in S\times T$, the intersection of the supports of $\mf^S_s$ and $\mf^T_t$ is of dimension 0, we have for all $j\geq1$
$$\Tor^j((id_{S\times T}\times p)^*\alpha_S^*((\mf^S)^{D}\otimes \tau_S^*\mo_{\p^2}(-1)),(id_{S\times T}\times q)^*\alpha_T^*(\mf^T\otimes \tau_{T}^*\mo_{\p^2}(2)))=0.$$
Then by Lemma 3.4 and Lemma A.7 (1) in \cite{Yuan5}, we have over $S\times T\times \p^2$
\[R^i(id_{S\times T}\times p)_{*}\widetilde{\F_{S,T}}=0,~\forall~i>0,~(id_{S\times T}\times p)_{*}\widetilde{\F_{S,T}}\cong \alpha_S^{*}\F_S^D\otimes\alpha_T^{*}\G_T,\]
and 
\[R^i(id_{S\times T}\times q)_{*}\widetilde{\F_{S,T}}=0,~\forall~i>0,~(id_{S\times T}\times q)_{*}\widetilde{\F_{S,T}}\cong \alpha_S^{*}\G_S^{\vee}\otimes\alpha_T^{*}\F_T.\]
Hence we have \begin{eqnarray}(\pi_{c^r_r}\times \pi_{u_d})^*\mathtt{B}_{c^r_r,u_d}&=&((\pi_{c^r_r}\circ\Phi\circ\kappa)\times (\pi_{u_d}\circ\Phi))^*\mathtt{D}_{c^d_d,u_r}\nonumber
\\
&=&\{(s,t)\big|H^0(\mathcal{D},p^{*}\F^D_{S,s}(-1)\otimes q^*\F_{T,t}(2))\neq 0,\}\end{eqnarray}
where $\mathtt{B}_{c^r_r,u_d}$ and $\mathtt{D}_{c^d_d,u_r}$ are as defined in (\ref{vsddiv}) and (\ref{sddiv}) respectively.

Therefore we have $\mathtt{B}_{c^r_r,u_d}=((\Phi\circ\kappa)\times (\Phi))^*\mathtt{D}_{c^d_d,u_r}$ and hence the diagram (\ref{SDtoVD}) commutes.
\end{proof}

\subsection{The proof of Theorem \ref{mainthm1}.}
As $r=n$, $M(r,0,n)\cong M(Q,(2r,r))$.  Theorem \ref{mainthm1} is essentially a corollary to Theorem \ref{mainquiver}.  The case $r=n=2$ has been proved in \cite{Yuan5}, hence without loss of generality, we assume $r\geq3$.


Recall that we have the following commutative diagram as in (\ref{comdia2})
\begin{equation}\label{comdia2p}\xymatrix@C=1cm{U(rH,0)\ar[r]^{\Psi\quad}\ar[d]_{\Phi}& M(Q,(r,r))\ar[d]^{g}_{\cong}\\ M(r,0,r)\ar[r]_{f\quad}^{\cong\quad}&M(Q,(r,2r))}.
\end{equation}

The two dimension vectors $(r,2r),~(d,d)$ of quiver $Q$ satisfy $\<(r,2r),(d,d)\>=0$ (def. see (\ref{debifm})).  Recall that there is a section $\bar{c}$ over $M(Q,(r,2r))\times M(Q,(d,d))$ defined in \S\ref{sdqr}.  

Define $U(Q,(d,d)):=\Psi(U(dH,0))$ and $V(Q,(r,2r))^b:=f(V(r,0,r)^b)$.  We then have the following isomorphism
\begin{equation}\label{sdcom1}U(dH,0)\times V(r,0,r)^b\xrightarrow[\cong]{(\Psi,f)}U(Q,(d,d))\times V(Q,(2r,r))^b.
\end{equation}
\begin{lemma}\label{pfmain3}Up to scalars, $(\Psi,f)^*\bar{c}=\varsigma_{c^r_r,u_d}$, where $\varsigma_{c^r_r,u_d}$ is as defined in Proposition \ref{vglob}.
\end{lemma}
\begin{proof}The proof is analogous to Claim 3.0.2 in \cite{Abe2} or Lemma 2.4 in \cite{LZ}.  Denote by $R_U$ ($R_V$, resp.) the preimage of $U(Q,(d,d))$ ($V(Q,(2r,r))^b$, resp.) inside $\Rep(Q,(d,d))$ ($\Rep(Q,(2r,r))$, resp.).  On $R_U$ and $R_V$ we have the universal representations as follows
\begin{equation}\label{ur1}\xymatrix@C=1cm{\ma_U\ar[r]^{Y_U}\ar@/^1pc/[r]^{X_U}\ar@/_1pc/[r]_{Z_U}&\mb_U},
\end{equation}
\begin{equation}\label{ur2}\xymatrix@C=1cm{\ma_V\ar[r]^{Y_V}\ar@/^1pc/[r]^{X_V}\ar@/_1pc/[r]_{Z_V}&\mb_U};
\end{equation}
where $\ma_U,\mb_U$ are rank $d$ bundles on $R_U$, 
$\ma_V$ is a rank $r$ bundle and $\mb_V$ is a rank $2r$ bundle on $R_V$.

On the other hand, on $\p^2\times R_U$ and $\p^2\times R_V$ we have the following two exact sequences of bundles.
\begin{equation}\label{us1}0\ra \ma_U\boxtimes\mo_{\p^2}(-2)\xrightarrow{x\cdot X_U+y\cdot Y_U+z\cdot Z_U}\mb_U\boxtimes\mo_{\p^2}(-1)\ra \F\ra0,
\end{equation}
\begin{equation}\label{us2}0\ra \ma_V\boxtimes\mo_{\p^2}(-2)\xrightarrow{x\cdot X_V+y\cdot Y_V+z\cdot Z_V}\mb_V\boxtimes\mo_{\p^2}(-1)\ra \G\ra0,;
\end{equation}
where $\F$ ($\G$ resp.) induces the map $R_U\xrightarrow{\pi_U} U(Q,(d,d))\xrightarrow{\Psi^{-1}} U(dH,0)$ ($R_V\xrightarrow{\pi_V} V(Q,(2r,r))^b\xrightarrow{f^{-1}} V(r,0,r)^b$ resp.).  Notice that $\G$ is locally free.  On $\p^2\times R_U\times R_V$ we have\tiny
\[0\ra\mathscr{H}om(\G,\F)\ra\begin{array}{c}\mathscr{H}om(\mb_V\boxtimes\mo_{\p^2}(-1),\mb_U\boxtimes\mo_{\p^2}(-1))\\ \oplus\\ \mathscr{H}om(\ma_V\boxtimes\mo_{\p^2}(-2),\ma_U\boxtimes\mo_{\p^2}(-2)) \end{array}\ra\mathscr{H}om(\ma_V\boxtimes\mo_{\p^2}(-2),\mb_U\boxtimes\mo_{\p^2}(-1))\ra 0,\] 
\normalsize
where by abuse of notation, we use the same letter to denote both the sheaf on $\p^2\times R_U$ or $\p^2\times R_V$ and its pull back to $\p^2\times R_U\times R_V$.  Define three projections $p_R:\p^2\times R_U\times R_V\ra R_U\times R_V$ and $p_V~(p_U):R_U\times R_V\ra R_U~(R_V)$.  Then on $R_U\times R_V$ we have\tiny
\[
(p_R)_*\left(\begin{array}{c}\mathscr{H}om(\mb_V\boxtimes\mo_{\p^2}(-1),\mb_U\boxtimes\mo_{\p^2}(-1))\\ \oplus\\ \mathscr{H}om(\ma_V\boxtimes\mo_{\p^2}(-2),\ma_U\boxtimes\mo_{\p^2}(-2)) \end{array}\right)\xrightarrow{d^V_U}
(p_R)_*\mathscr{H}om(\ma_V\boxtimes\mo_{\p^2}(-2),\mb_U\boxtimes\mo_{\p^2}(-1))\]
\[\implies\begin{array}{c}\mathscr{H}om(p_V^*\mb_V,p_U^*\mb_U)\\ \oplus\\ \mathscr{H}om(p_V^*\ma_V,p_U^*\ma_U)\end{array} \xrightarrow{d^V_U}\mathscr{H}om(p_V^*\ma_V,p_U^*\mb_U)^{\oplus3}.\]
\normalsize
The map $d^V_U$ above can be represented by a square matrix of order $3rd$.  The function $c$ on $R_U\times R_V$ equals to $\det(d^V_U)$ by \S\ref{semi}.  On the other hand by Lemma \ref{VSD1} and Proposition \ref{vglob}, $\det(d^V_U)$ is (up to scalars) the pull back of $\varsigma_{c^r_r,d}$ to $R_U\times R_V$ via the map $(\Psi^{-1}\circ \pi_U, f^{-1}\circ\pi_V)$.  
Hence the lemma. 
\end{proof}
\begin{proof}[Proof of Theorem \ref{mainthm1}]As we have seen in Theorem \ref{mainquiver} in \S\ref{sdqr}, the section $\bar{c}$ induces the following isomorphism
\footnotesize
\[SD(Q):H^0(M(Q,(r,2r)),\lambda(Q,(r,2r))_{-\<-,(d,d)\>})^{\vee}\ra H^0(M(Q,(d,d)),\lambda(Q,(d,d))_{\<(r,2r),-\>}),\]
\normalsize

Since $f$ and $g$ in (\ref{comdia2p}) are isomorphisms, $M(Q,(r,2r))\setminus V(Q,(r,2r))^b$ ($M(Q,(d,d))\setminus U(Q,(d,d))$ resp.) is of codimension $\geq 2$ inside the irreducible normal scheme $M(Q,(r,2r))$ ($M(Q,(d,d))$ resp.).  Hence by Proposition \ref{VDtoSD} and Lemma \ref{pfmain3}, we have the following commutative diagram
\begin{equation}\label{pfcomdia1}\xymatrix{H^0(M(Q,(r,2r)),\lambda(Q,(r,2r))_{-\<-,(d,d)\>})^{\vee}\ar[r]^{\quad SD(Q)}_{\cong} &H^0(M(Q,(d,d)),\lambda(Q,(d,d))_{\<(r,2r),-\>})\ar[d]_{\cong}^{restr.}\\
H^0(V(Q,(r,2r))^b,\lambda(Q,(r,2r))_{-\<-,(d,d)\>})^{\vee}\ar[u]_{\cong}^{(restr.)^{\vee}}\ar[r]^{\quad SD(Q)} &H^0(U(Q,(d,d)),\lambda(Q,(d,d))_{\<(r,2r),-\>})\ar[d]_{\cong}^{\Phi^*}\\
H^0(V(r,0,r)^b,\lambda_{c^r_r}(u_d))^{\vee}\ar[u]_{\cong}^{(f^*)^{\vee}}\ar[d]^{\cong}_{(\Phi^*)^{\vee}}\ar[r]^{VD_{c^r_r,d}}& H^0(U(dH,0),\lambda_d(c^r_r))\ar[d]^{\Phi^*}_{\cong}\\
H^0(U(rH,0)^b,\lambda_r(c^d_d))^{\vee}\ar[d]^{(restr.)^{\vee}}_{\cong}\ar[r]^{SD_{u_d,c^r_r}\circ(\kappa^*)^{\vee}}& H^0(V(d,0,d),\lambda_{c^d_d}(u_d))\\
H^0(M(rH,0),\lambda_{r}(c^d_d))^{\vee}\ar[r]^{SD_{u_d,c^r_r}\circ(\kappa^*)^{\vee}}& H^0(M(d,0,d),\lambda_{c_d^d}(u_r))\ar[u]^{\cong}_{restr.}
}
\end{equation}
$\kappa^{*}$ is an isomorphism by Corollary A.5 in \cite{Yuan5}.  By (\ref{pfcomdia1}) we get directly that $SD_{c^r_r,d}$ is an isomorphism and hence the theorem.
\end{proof}

\subsection{Some useful lemmas.}
In this subsection, we want to study the function $c$ on $\Rep(Q,(r,2r))\times\Rep(Q,(d,d))$ and $\bar{c}$ on $M(Q,(r,2r))\times M(Q,(d,d))$, as defined in \S\ref{semi} and \S\ref{sdqr}.  

Let $\Mat(m\times n,\bc)$ be the set of all $m\times n$ matrices with entries in $\bc$.  For any two matrices $\Gamma=\{\gamma_{ij}\}\in \Mat(m\times n,\bc),~\Omega=\{\omega_{st}\}\in \Mat(k\times l,\bc)$, we define
\begin{equation}\label{stardef}\Mat(mk\times nl,\bc)\ni~\Gamma\ast\Omega:=\left(\begin{array}{cccc}\omega_{11}\cdot\Gamma,&\omega_{12}\cdot\Gamma,&\cdots,&\omega_{1l}\cdot\Gamma\\\omega_{21}\cdot\Gamma,&\omega_{22}\cdot\Gamma,&\cdots,&\omega_{1l}\Gamma\\
\vdots&\ddots&\ddots&\vdots\\ \omega_{k1}\cdot\Gamma,&\omega_{k2}\cdot\Gamma,&\cdots,&\omega_{1l}\cdot\Gamma\end{array}\right)
\end{equation}
The following lemma is easy to see.
\begin{lemma}\label{ast}
Let $\Gamma,\Gamma_i\in \Mat(m\times n,\bc),~\Omega,\Omega_i\in \Mat(k\times l,\bc)$ $(i=1,2,\cdots,p)$.  Let $\Pi\in Mat(l\times h,\bc)$ and $\Delta\in Mat(n\times b,\bc)$ for any $h,b\in\mathbb{Z}_{>0}$.  Then for the operator $\ast$ we have the following three properties:
\begin{enumerate}
\item If $m=n,~k=l$, and $\Gamma,\Omega$ are invertible, then $\Gamma\ast\Omega$ is invertible and $(\Gamma\ast\Omega)^{-1}=(\Gamma)^{-1}\ast (\Omega)^{-1}$;
\item $(\Gamma\cdot \Delta)\ast(\Omega\cdot \Lambda)=(\Gamma\ast\Omega)\cdot(\Delta\ast\Lambda)$;
\item If $mk=nl$, then $$\det(\Sigma_{i=1}^p\Gamma_i\ast\Omega_i)=(-1)^{\frac{mk((m-1)(k-1)+(n-1)(l-1))}4}\det(\Sigma_{i=1}^p\Omega_i\ast\Gamma_i).$$
\end{enumerate}
\end{lemma}

Let $(V,W)\in \Rep(Q,(r,2r))^{ss}\times\Rep(Q,(d,d))^{ss}$, and let $V$ ($W$, resp.) be represented by three $r\times 2r$ ($d\times d$, resp.) matrices $(\widetilde{A}_x^r,\widetilde{A}_y^r,\widetilde{A}_z^r)$ ($(B_x^d,B_y^d,B_z^d)$, resp.).  Let $g^{-1}(V)$ be represented by three $r\times r$ matrices $(A_x^r,A_y^r,A_z^r)$ well-defined up to the action of $GL(Q,(r,r))$.
Define \begin{equation}\label{matdef}\Mat(dr\times dr,\bc)\ni C(V,W):=B_x^d\ast A_x^r+B_y^d\ast A_y^r+B_z^d\ast A_z^r.\end{equation}


\begin{lemma}\label{exdet}
\begin{enumerate}
\item $c(V,W)=\det\begin{pmatrix}B_x^d\ast\bi_{r},& -\bi_d\ast\widetilde{A}_x^r\\ B_y^d\ast\bi_{r},& -\bi_d\ast\widetilde{A}_y^r\\ B_z^d\ast\bi_{r},& -\bi_d\ast\widetilde{A}_z^r\end{pmatrix}$;
\item Up to scalars, for any $(V,W)\in M(Q,(r,2r))\times M(Q,(d,d))$, $\bar{c}(V,W)=\det(C(V,W))$ with $C(V,W)$ defined in (\ref{matdef}).
\end{enumerate}
\end{lemma}
\begin{proof}(1) is obtained by the definition of $c$.

By (\ref{compare11}), we have 
 \begin{equation}\label{usecom11}\widetilde{P}\cdot\begin{pmatrix}\widetilde{A}^r_x,&x\bi_{r}\\ \widetilde{A}^r_y,&y\bi_{r}\\ \widetilde{A}^r_z,& z\bi_{r}\end{pmatrix}=\begin{pmatrix}\bi_{2r},&*\\ \mathbf{0}_{r\times 2r}, & x A^r_x+yA^r_y+zA^r_z\end{pmatrix},
\end{equation} 
where $\widetilde{P}\in GL(3r)$.  Hence
 \begin{equation}\label{usecom12}\begin{pmatrix}\widetilde{A}^r_x\\ \widetilde{A}^r_y\\ \widetilde{A}^r_z\end{pmatrix}=\widetilde{P}^{-1}\cdot\begin{pmatrix}\bi_{2r}\\ \mathbf{0}_{r\times 2r}\end{pmatrix},~\text{and}~~\begin{pmatrix}x\bi_{r}\\ y\bi_r\\ z\bi_r\end{pmatrix}=\widetilde{P}^{-1}\cdot\begin{pmatrix}*\\  x A^r_x+yA^r_y+zA^r_z\end{pmatrix},
\end{equation} 

By Lemma \ref{ast} (2), we have
\begin{eqnarray}\label{usecom13}\begin{pmatrix}\bi_d\ast\widetilde{A}_x^r\\  \bi_d\ast\widetilde{A}_y^r\\  \bi_d\ast\widetilde{A}_z^r\end{pmatrix}
&=&\bi_d\ast\begin{pmatrix}\widetilde{A}^r_x\\ \widetilde{A}^r_y\\ \widetilde{A}^r_z\end{pmatrix}=\bi_d\ast(\widetilde{P}^{-1}\cdot\begin{pmatrix}\bi_{2r}\\ \mathbf{0}_{r\times 2r}\end{pmatrix})\nonumber\\
&=&(\bi_d\ast\widetilde{P}^{-1})\cdot (\bi_d\ast\begin{pmatrix}\bi_{2r}\\ \mathbf{0}_{r\times 2r}\end{pmatrix})=(\bi_d\ast\widetilde{P}^{-1})\cdot\begin{pmatrix}\bi_{2rd}\\ \mathbf{0}_{rd\times 2rd}\end{pmatrix}.
\end{eqnarray}

By (\ref{usecom12}) we have $\exists ~\widetilde{P}_1\in \Mat(2r\times 3r,\bc)$, such that $\widetilde{P}=\begin{pmatrix}\widetilde{P}_1\\ A^r_x,A^r_y,A^r_z\end{pmatrix}$.  Hence
\begin{eqnarray}\label{usecom14}(\bi_d\ast\widetilde{P})\cdot\begin{pmatrix}B_x^d\ast\bi_r\\  B_y^d\ast\bi_r\\  B_z^d\ast\bi_r\end{pmatrix}
&=&\begin{pmatrix}\bi_d\ast\widetilde{P}_1\\ \bi_d\ast A^r_x,\bi_d\ast A^r_y,\bi_d\ast A^r_z\end{pmatrix}\cdot\begin{pmatrix}B_x^d\ast\bi_r\\  B_y^d\ast\bi_r\\  B_z^d\ast\bi_r\end{pmatrix}\nonumber\\ 
&=&\begin{pmatrix}(\bi_d\ast\widetilde{P}_1)\cdot \begin{pmatrix}B_x^d\ast\bi_r\\  B_y^d\ast\bi_r\\  B_z^d\ast\bi_r\end{pmatrix}\\ B_x^d\ast A^r_x+B_y^d\ast A^r_y+B_z^d\ast A_z^r\end{pmatrix}\end{eqnarray}

Therefore we have\scriptsize
\begin{eqnarray}\det\begin{pmatrix}B_x^d\ast\bi_{r},& -\bi_d\ast\widetilde{A}_x^r\\ B_y^d\ast\bi_{r},& -\bi_d\ast\widetilde{A}_y^r\\ B_z^d\ast\bi_{r},& -\bi_d\ast\widetilde{A}_z^r\end{pmatrix}&=&\det (\bi_d\ast\widetilde{P}^{-1})\cdot \det\begin{pmatrix}(\bi_d\ast\widetilde{P}_1)\cdot \begin{pmatrix}B_x^d\ast\bi_r\\  B_y^d\ast\bi_r\\  B_z^d\ast\bi_r\end{pmatrix},&-\bi_{2rd}\\B_x^d\ast A^r_x+B_y^d\ast A^r_y+B_z^d\ast A_z^r ,&\mathbf{0}_{rd\times 2rd}\end{pmatrix}\nonumber\\ 
&\sim&\det(B_x^d\ast A^r_x+B_y^d\ast A^r_y+B_z^d\ast A_z^r).\end{eqnarray}\normalsize
Hence we have proved (2).

\end{proof}

\subsection{The proof of Theorem \ref{mainthm2}.}
From now on let $n>r\geq 2$.  We at first define closed subsets $\ts_{n}^i\subset M(Q,(n,2n))\cong M(n,0,n)$ as follows.  \begin{equation}\label{defS}\Rep(Q,(n,2n))^{ss}\supset\widetilde{\ts}_{n}^i:=\{
V\big|rank(C(V,\Lambda_3))\leq 3n-i\};\quad \ts_n^i:=\widetilde{\ts}_{n}^i//GL(Q,(n,2n));
\end{equation} 
where $\Lambda_3\in\Rep(Q,(3,3))$ is defined in Remark \ref{extend1} and $C(V,\Lambda_3)$ is defined in (\ref{matdef}).  East to see that $\widetilde{\ts}_n^i$ is a $GL(Q,(n,2n))$-invariant closed subscheme of $\Rep(Q,(n,2n))^{ss}$ and hence $\ts_n^i$ is a well-defined closed subscheme of $M(Q,(n,2n))$.  

Let $V\in\Rep(Q,(n,2n))^{ss}$ be represented by $n\times 2n$ matrices $(\widetilde{A}_x^V,\widetilde{A}_y^V,\widetilde{A}_z^V)$, then we have the following exact sequence
\begin{equation}\label{qtos}0\ra\mo_{\p^2}(-2)^{\oplus n}\xrightarrow{x\cdot\widetilde{A}_x^V+y\cdot\widetilde{A}_y^V+z\cdot\widetilde{A}_z^V}\mo_{\p^2}(-1)^{\oplus 2n}\ra\mg_n^n(V)\ra0,\end{equation}
where $\mg_n^n(V)$ is semistable.

\begin{lemma}\label{dimsec}For any $V\in\Rep(Q,(n,2n))^{ss}$, let $\mg_n^n(V)$ be the same as in (\ref{qtos}).  Then we have
$$V\in\widetilde{\ts}_n^i\Leftrightarrow\emph{hom}(\mg_n^n(V),\mo_{\p^2})\geq i.$$
In particular, we identify $M(n,0,n)$ with $M(Q,(n,2n))$ and have 
$$\mg\in\ts_n^i\Leftrightarrow\emph{hom}(\mg,\mo_{\p^2})\geq i.$$
\end{lemma}
\begin{proof}By (\ref{qtos}) we have 
\[\text{hom}(\mg_n^n(V),\mo_{\p^2})=6n-rank\begin{pmatrix}\widetilde{A}_x^V,&0,&0\\ \widetilde{A}_y^V,&\widetilde{A}_x^V,&0\\ \widetilde{A}_z^V,&0,&\widetilde{A}_x^V\\ 0,&\widetilde{A}_y^V,&0\\0,&\widetilde{A}_z^V,&\widetilde{A}_y^V\\0,&0,&\widetilde{A}_z^V\end{pmatrix}\]

Recall that $\Lambda_3$ can be represented by matrices $(A^{\Lambda}_{x},A^{\Lambda}_{y},A^{\Lambda}_{z})$ such that $x\cdot A^{\Lambda}_{x}+y\cdot A^{\Lambda}_{y}+z\cdot A^{\Lambda}_{z} =\begin{pmatrix}y,&-z,&0\\-x,&0,&z\\0,&x,&-y\end{pmatrix}$.

We then have 
\begin{eqnarray}\begin{pmatrix}A_x^{\Lambda}\ast\bi_{n},& -\bi_3\ast\widetilde{A}_x^V\\ A_y^{\Lambda}\ast\bi_{n},& -\bi_3\ast\widetilde{A}_y^V\\ A_z^{\Lambda}\ast\bi_{n},& -\bi_3\ast\widetilde{A}_z^V\end{pmatrix}&=&\begin{pmatrix}0,&0,&0,&-\widetilde{A}_x^V,&0,&0\\ -\bi_n,&0,&0,&0,&-\widetilde{A}_x^V,&0\\ 0,&\bi_n,&0,&0,&0,&-\widetilde{A}_x^V\\ \bi_n,&0,&0,&-\widetilde{A}_y^V,&0,&0\\ 0,&0,&0,&0,&-\widetilde{A}_y^V,&0\\0,&0,&-\bi_n,&0,&0,&-\widetilde{A}_y^V\\ 0,&-\bi_n,&0,&-\widetilde{A}_z^V,&0,&0\\ 0,&0,&\bi_n,&0,&-\widetilde{A}_z^V,&0\\0,&0,&0,&0,&0,&-\widetilde{A}_z^V\end{pmatrix}\nonumber\\
&\simeq &\begin{pmatrix}\bi_{3n},&0\\0,&-\begin{pmatrix}\widetilde{A}_x^V,&0,&0\\ \widetilde{A}_y^V,&\widetilde{A}_x^V,&0\\ \widetilde{A}_z^V,&0,&\widetilde{A}_x^V\\ 0,&\widetilde{A}_y^V,&0\\0,&\widetilde{A}_z^V,&\widetilde{A}_y^V\\0,&0,&\widetilde{A}_z^V\end{pmatrix}\end{pmatrix} \end{eqnarray}

On the other hand, by (\ref{usecom13}) and (\ref{usecom14}) we have
\[\begin{pmatrix}A_x^{\Lambda}\ast\bi_{n},& -\bi_3\ast\widetilde{A}_x^V\\ A_y^{\Lambda}\ast\bi_{n},& -\bi_3\ast\widetilde{A}_y^V\\ A_z^{\Lambda}\ast\bi_{n},& -\bi_3\ast\widetilde{A}_z^V\end{pmatrix}\simeq \begin{pmatrix}(\bi_3\ast\widetilde{P}_1)\cdot \begin{pmatrix}A_x^{\Lambda}\ast\bi_n\\  A_y^{\Lambda}\ast\bi_n\\  A_z^{\Lambda}\ast\bi_n\end{pmatrix},&-\bi_{6n}\\C(V,\Lambda_3) ,&\mathbf{0}_{3n\times 6n}\end{pmatrix}. \]
Hence we have
\[rank\begin{pmatrix}\widetilde{A}_x^V,&0,&0\\ \widetilde{A}_y^V,&\widetilde{A}_x^V,&0\\ \widetilde{A}_z^V,&0,&\widetilde{A}_x^V\\ 0,&\widetilde{A}_y^V,&0\\0,&\widetilde{A}_z^V,&\widetilde{A}_y^V\\0,&0,&\widetilde{A}_z^V\end{pmatrix}+3n=rank\begin{pmatrix}A_x^{\Lambda}\ast\bi_{n},& -\bi_3\ast\widetilde{A}_x^V\\ A_y^{\Lambda}\ast\bi_{n},& -\bi_3\ast\widetilde{A}_y^V\\ A_z^{\Lambda}\ast\bi_{n},& -\bi_3\ast\widetilde{A}_z^V\end{pmatrix}=rank (C(V,\Lambda_3))+6n,\]
which is equivalent to $\text{hom}(\mg_n^n(V),\mo_{\p^2})+rank(C(V,\Lambda_3))=3n$.  Therefore we proved the lemma.\end{proof}

We want to construct a birational map $\delta_{i}:M(n-i,0,n)\ra\ts_n^i$ which generalizes the map $\delta$ in Proposition 3.1 in \cite{Yuan7}.  Firstly we have the following lemma which generalizes Lemma 2.12 in \cite{Yuan7}.
\begin{lemma}\label{lconsec}Let $\mg_n^r$ be of class $c_n^r$ and we have the following exact sequence
\begin{equation}\label{consec}0\ra\mg_n^r\ra\mg_n^n\ra\mo_{\p^2}^{\oplus (n-r)}\ra 0.\end{equation}
Then we have
\begin{enumerate}
\item If $\mg_n^r$ is $\mu$-semistable with $H^0(\mg_n^r)=0$, then up to isomorphism there is a unique $\mg_n^n$ such that the sequence (\ref{consec}) does not partially split; and in this case the extension $\mg_n^n$ is semistable and lies in $\ts_n^{n-r}$.  If moreover $\mg_n^r$ is $\mu$-stable, then $\mg_n^n$ is stable.
\item For every $'\mg_n^n\in \ts_n^{n-r}$, we can find a semistable sheaf $\mg_n^n$ $S$-equivalent to $'\mg_n^n$ which lies in the sequence (\ref{consec}) with $\mg_n^r$ $\mu$-semistable and $H^0(\mg_n^r)=0$. 
\end{enumerate}
\end{lemma}
\begin{proof}If $\mg_n^r$ is $\mu$-semistable, then $H^2(\mg_n^r)=0$ and $\mg_n^n$ is also $\mu$-semistable.  Hence $\mg_n^n$ is semistable iff $H^0(\mg_n^n)=0$.  We have $\Ext^1(\mo_{\p^2},\mg_n^r)\cong H^1(\mg_n^r)$ is of dimension $n-r$ since $H^0(\mg_n^r)=H^2(\mg_n^r)=0$, hence there is a unique up to isomorphisms extension $\mg_n^n$ such that (\ref{consec}) does not partially split.  Easy to see that in this case $H^0(\mg_n^n)=0$.  

Assume $\mg_n^r$ is $\mu$-stable.  Then for every non-trivial quotient $\mg_n^n\twoheadrightarrow\mq$, either $H^0(\mq)\neq 0$ or $c_1(\mq).H>0$.  Hence $\mg_n^n$ can not be strictly semistable.  Hence Statement (1).

On the other hand let $\mg_n^n\in M(n,0,n)^s\cap\ts^i_r$, define the map $h:\mg_n^n\ra\mo_{\p^2}^{i}$ whose restriction to any direct summand is not zero ($h$ is given by an $i$-dimensional subspace of $\Hom(\mg_n^n,\mo_{\p^2})$).  Denote by $Im(h)$ the image of $h$.  Then $Im(h)$ is $\mu$-semistable.  By stability of $\mg_n^n$ we have $\chi(Im(h))>0$ and hence $H^0(Im(h))\neq 0$.  Therefore we have an injection $\mo_{\p^2}\hookrightarrow Im(h)$.  However $Im(h)/\mo_{\p^2}$ is also a quotient of $\mg_n^n$ and hence $\chi(Im(h)/\mo_{\p^2})>0$.  By induction we have $Im(h)\cong\mo_{\p^2}^{\oplus i}$ and $h$ is surjective.

In general for $'\mg_n^n$ semistable, we choose $\mg_n^n=\bigoplus_{k=1}^l\mg_{n_k}^{n_k}$ with $\mg_{n_k}^{n_k}$ stable which is $S$-equivalent to $'\mg_n^n$ and then we can get a sequence (\ref{consec}).
\end{proof}
\begin{rem}\label{iton}Combine Lemma \ref{dimsec} and Lemma \ref{lconsec}, we have $\widetilde{\ts}_n^i$ is empty for $i\geq n$.
\end{rem}

Denote by  $\km(r,0,n)$ ($\km(r,0,n)^{s}$, $\km(r,0,n)^{\mu}$, $\km(r,0,n)^{\mu s}$, resp.) the stack of semistable (stable, $\mu$-semistable, $\mu$-stable, resp.) sheaves of class $c^r_n$.  Denote by $\ks_n^i$ the preimage of $\ts_n^i$ inside $\km(r,0,n)$.  By Lemma \ref{lconsec}, we have a rational map $\widetilde{\delta}_{n-r}:\km(r,0,n)^{\mu}\ra\ks_n^{n-r}$ inducing a surjective map $\km(r,0,n)^{\mu}\ra\ks_n^{n-r}\ra\ts_n^{n-r}$.  Easy to see that $\widetilde{\delta}_{n-r}(\km(r,0,n)^{\mu s})\subset \km(n,0,n)^s\cap(\ks_n^{n-r}\setminus\ks_n^{n-r+1})$ and $\widetilde{\delta}_{n-r}$ restricted to $\km(r,0,n)^{\mu s}$ is an isomorphism to its image.

\begin{lemma}\label{serres2}Let $0\leq i\leq n-1$.  Then $\widetilde{\ts}_n^i\setminus\widetilde{\ts}_n^{i+1}$ is a locally complete intersection in $\Rep(Q,(n,2n))^{ss}$ with codimension $i^2$.
\end{lemma}
\begin{proof}By definition $\widetilde{\ts}_n^i\setminus\widetilde{\ts}_n^{i+1}$ can be defined by $i^2$ equations inside $\Rep(Q,(n,2n))^{ss}$.  Since $\Rep(Q,(n,2n))^{ss}$ is open in $\Rep(Q,(n,2n))$ by Theorem \ref{GITK}, we have $\dim(\Rep(Q,(n,2n))^{ss})=6n^2$.  We only need to show that 
$$\dim(\widetilde{\ts}_n^i\setminus\widetilde{\ts}_n^{i+1})\leq 6n^2-i^2.$$
By Lemma \ref{lconsec}, we have
 \begin{eqnarray}\label{dimS}\dim(\Rep(Q,(n,2n))^s\cap(\widetilde{\ts}_n^i\setminus\widetilde{\ts}_n^{i+1}))&\leq &\dim(\km(n-i,0,n)^{\mu})+\dim(GL(Q,(n,2n)))\nonumber\\
 &=&(n-i)(n+i)+5n^2=6n^2-i^2.\end{eqnarray}
On the other hand, for any $V_k\in \Rep(Q,(n_k,2n_k))$ with $k=1,2$, we have $\text{ext}^1_Q(V_1,V_2)\leq 6n_1n_2$ by (\ref{qext1}).  Hence by induction assumption on $n$, strictly semistable points in $\widetilde{\ts}_n^i\setminus\widetilde{\ts}_n^{i+1}$ form a closed subset of dimension no more than $\displaystyle{\max_{\substack{0\leq i_k\leq n_k, k=1,2\\ n_1+n_2=n\\i_1+i_2=i}}}\{6n_1^2+6n_2^2-i_1^2-i_2^2+6n_1n_2\}\leq 6n^2-i^2-4(n-1).$

Therefore we proved the lemma.
\end{proof}
\begin{lemma}\label{Snormal}Both $\widetilde{\ts}_n^i$ and $\ts_n^i$ are normal, for $0\leq i\leq n-1$.
\end{lemma}
\begin{proof}It is enough to show $\widetilde{\ts}_n^i$ is normal since $\ts_n^i$ is a good quotient of $\widetilde{\ts}_n^i$.

By Lemma \ref{serres2}, we see that $\widetilde{\ts}_n^i$ is Cohen-Macaulay with an open dense subset $\Rep(Q,(n,2n))^s\cap(\widetilde{\ts}_n^i\setminus\widetilde{\ts}_n^{i+1})$.  The complement of $\Rep(Q,(n,2n))^s\cap(\widetilde{\ts}_n^i\setminus\widetilde{\ts}_n^{i+1})$ inside 
$\widetilde{\ts}_n^i$ is of codimension $\geq2$.
 
The scheme $\Rep(Q,(n,2n))^s\cap(\widetilde{\ts}_n^i\setminus\widetilde{\ts}_n^{i+1})$ is a $GL(Q,(n,2n))$-bundle over $\km(n,0,n)^s\cap(\ks_n^{n-r}\setminus\ks_n^{n-r+1})$.  By Lemma 2.10 in \cite{Yuan7}, we have $\widetilde{\delta}(\km(r,0,n)^{\mu s})\cong \km(r,0,n)^{\mu s}$ is a smooth irreducible dense open substack of $\km(n,0,n)^s\cap(\ks_n^{n-r}\setminus\ks_n^{n-r+1})$ whose complement is of codimension $\geq 2$.  Hence $\Rep(Q,(n,2n))^s\cap(\widetilde{\ts}_n^i\setminus\widetilde{\ts}_n^{i+1})$ is irreducible and regular in codimension 1, hence so is $\widetilde{\ts}_n^i$.
\end{proof}

\begin{prop}\label{genmap}We have a birational morphism $\delta_{n-r}:M(r,0,n)\ra\ts_n^i$.  This morphism $\delta_{n-r}$ induces an isomorphism from  $M(r,0,n)^{\mu s}$ to its image which is contained in $M(n,0,n)^s\cap(\ts_n^{n-r}\setminus\ts_n^{n-r+1})$.  Here $M(r,0,n)^{\mu s}$ consists of all $\mu$-stable sheaves in $M(r,0,n)$.  

Moreover $\delta_{n-r}^*\lambda_{c^n_n}(d)\cong\lambda_{c^r_n}(d)$ for all $d\in\bz$ and hence we have the following isomorphism
\begin{equation}\label{resnr}\delta_{n-r}^*:H^0(\ts_n^{n-r},\lambda_{c_n^n}(d))\xrightarrow{\cong} H^0(M(r,0,n),\lambda_{c_n^r}(d)).
\end{equation}
\end{prop}
\begin{proof}Given Lemma \ref{Snormal}, the statement $\delta_{n-r}^*\lambda_{c^n_n}(d)\cong\lambda_{c^r_n}(d)$ is the only thing left to prove, which can be deduced from analogous argument to Proposition 3.1 in \cite{Yuan7}. 
\end{proof}
\begin{lemma}\label{surres}The restriction map
\begin{equation}\label{gamma}\gamma_n^{n-r}:H^0(M(n,0,n),\lambda_{c^n_n}(d))\rightarrow H^0(\ts_n^{n-r},\lambda_{c^n_n}(d)) \end{equation}
is surjective.
\end{lemma}
\begin{proof}Since $GL(Q,(n,2n))$ is reductive and the ideal sheaf $\mi_{\widetilde{\ts}_n^{n-r}}$ of $\widetilde{\ts}_n^{n-r}$ is $GL(Q,(n,2n))$-invariant, the restriction map
$$ \mathbb{C}[\Rep(Q,(n,2n))]^{GL(Q,(n,2n)),-\<-,(d,d)\>}\ra(\mathbb{C}[\Rep(Q,(n,2n))]/\mi_{\widetilde{\ts}_n^{n-r}})^{GL(Q,(n,2n)),-\<-,(d,d)\>}$$
is surjective by the basic fact of the reductive group over $\mathbb{C}$ (see e.g. Fact (3) pp.29 in \cite{MFK}).  

We have the following commutative diagram
$$\xymatrix{\mathbb{C}[\Rep(Q,(n,2n))]^{GL(Q,(n,2n)),-\<-,(d,d)\>}\ar[d]^{\cong}\ar[r]&(\mathbb{C}[\Rep(Q,(n,2n))]/\mi_{\widetilde{\ts}_n^{n-r}})^{GL(Q,(n,2n)),-\<-,(d,d)\>}\ar[d]^{\cong}\\
H^0(M(n,0,n),\lambda_{c^n_n}(d))\ar[r]^{\gamma_n^{n-r}}&H^0(\ts_n^{n-r},\lambda_{c^n_n}(d)). }$$ 
Hence $\gamma_n^{n-r}$ in (\ref{gamma}) is surjective. 
\end{proof}

\begin{prop}\label{gpfmain3}We have the following commutative diagram
\begin{equation}\label{cdmain3}\xymatrix{H^0(M(r,0,n),\lambda_{c^r_n}(d))^{\vee}\ar[r]_{\cong}^{\quad(\delta_{n-r}^*)^{\vee}}\ar[d]_{SD_{c^r_n,d}}& H^0(\ts_n^{n-r},\lambda_{c^n_n}(d))^{\vee}\ar[d]^{\alpha_{\ts_n^{n-r}}}\ar@{^{(}->}[r]^{(\gamma_n^{n-r})^{\vee}\quad} &H^0(M(n,0,n),\lambda_{c^n_n}(d))^{\vee}\ar[d]_{\cong}^{SD_{c_n^n,d}}\\ H^0(M(dH,0),\lambda_d(c_n^r))\ar[r]_{Id}^{\cong} &H^0(M(dH,0),\lambda_d(c_n^r))\ar@{^{(}->}[r]_{\jmath_n^r}^{.\theta_d^{n-r}}& H^0(M(dH,0),\lambda_d(c_n^n)),}
\end{equation}
where $\jmath_n^r$ is as defined in (\ref{mzeta}).
%
%
\end{prop}
\begin{proof}The proof is analogous to that of Proposition 4.1 in \cite{Yuan7} and hence is omitted here.
\end{proof}
\begin{proof}[Proof of Theorem \ref{mainthm2}]The theorem follows straightforward from Proposition \ref{gpfmain3} and Theorem \ref{mainthm1}.
\end{proof}

\begin{rem}For every sheaf $\mg$ of class $c^r_n$, we can define a divisor $D_{\mg}:=\{\mf\in M(dH,0)\big| H^0(\mg\otimes\mf)\ne 0\}$.
If $D_{\mg}\neq M(dH,0)$, then up to scalars it gives a section $s_{\mg}$ of line bundle $\lambda_{d}(c^r_n)$.  Theorem \ref{dewe} actually implies that $\{s_{\mg}\}_{\mg\in M(n,0,n)}$ spans $H^0(M(dH,0),\lambda_{d}(c^n_n))$, i.e. $SD_{c^n_n,d}$ is effectively surjective (see Definition 4.12 in \cite{Yuan7}).  However, we don't know at now whether $SD_{c^r_n,d}$ is surjective or not for $r\leq n-2$.  If one could show that every section $s$ of the form $s'\cdot\theta_d^{n-r}$ can be written as a linear combination of $s_{\mg}$ with $\mg\in\ts_n^{n-r}$, then it would follow that $SD_{c^r_n,d}$ were also effectively surjective and hence an isomorphism.    
\end{rem}
\begin{rem}One may try to extend the strategy in this paper to other rational surfaces, such as Hirzebruch surfaces.  But it is not easy.  The tilting theory is in general much more complicated on other rational surfaces.  As a result it is difficult to get analogs of Theorem \ref{mainthm1} from the result of Derksen and Weyman (Theorem \ref{dewe}).  However, once we got an analog of Theorem \ref{mainthm1}, we would also get an analog of Theorem \ref{mainthm2} by the same routine.
\end{rem}

\begin{flushleft}{\textbf{Acknowledgments.}} I was supported by NSFC 11771229.  I would like to thank Jiarui Fei for his help on quiver representation theory. 
\end{flushleft}

\end{document}